\documentclass[a4paper,10pt,reqno]{amsart}
\usepackage{amsmath,amsfonts,amscd,amssymb,graphicx,mathrsfs,eufrak,mathrsfs}

\usepackage[nospace,noadjust]{cite}
\usepackage[dvips,all,arc,curve,color,frame]{xy}
\usepackage{tikz}
\usetikzlibrary{arrows,decorations.pathmorphing,decorations.pathreplacing,positioning,shapes.geometric,shapes.misc,decorations.markings,decorations.fractals,calc,patterns}
\usepackage[colorlinks,hypertexnames=false]{hyperref}
\usepackage{comment}

\tikzset{>=stealth',
        cvertex/.style={circle,draw=black,inner sep=1pt,outer sep=3pt},
        vertex/.style={circle,fill=black,inner sep=1pt,outer sep=3pt},
        star/.style={circle,fill=yellow,inner sep=0.75pt,outer sep=0.75pt},
        tvertex/.style={inner sep=1pt,font=\scriptsize},
        gap/.style={inner sep=0.5pt,fill=white}}

\usepackage{mathtools}

\author{J. Karmazyn}
\title{Quiver GIT for Varieties with Tilting Bundles}

\theoremstyle{plain} 
\newtheorem{VitalLemma1}{Lemma}[subsection]
\newtheorem{VitalLemma2}[VitalLemma1]{Lemma}

\newtheorem{VitalLemma3}[VitalLemma1]{Lemma}

\newtheorem{DerivedBaseChange}[VitalLemma1]{Lemma}
\newtheorem{DerivedBaseChangeCor}[VitalLemma1]{Corollary}
\newtheorem{BigTheorem1}[VitalLemma1]{Corollary}

\newtheorem{BigTheorem3}[VitalLemma1]{Theorem}
\newtheorem{ssModuliGIT}[VitalLemma1]{Theorem}
\newtheorem{sModuliGIT}[VitalLemma1]{Theorem}

\newtheorem{Surjectivity}[VitalLemma1]{Lemma}
\newtheorem{Representability}[VitalLemma1]{Theorem}
\newtheorem{TiltingBundles1dim}[VitalLemma1]{Theorem}

\newtheorem*{pBigTheorem3}{Theorem}
\newtheorem*{pBigTheorem1}{Corollary}

\newtheorem*{pRepresentability}{Theorem}
\newtheorem*{pFlopGIT}{Corollary}
\newtheorem*{SekiyaYamaura}{Theorem}
\newtheorem*{pRationalSurfaceSing}{Corollary}
\newtheorem{pMainTheoremCor}[VitalLemma1]{Corollary}
\newtheorem{ReconstructionAlg}[VitalLemma1]{Corollary}
\newtheorem{Tilting}[VitalLemma1]{Theorem}
\newtheorem{VanDenBerghFlops}[VitalLemma1]{Theorem}
\newtheorem{FlopGIT}[VitalLemma1]{Corollary}
\newtheorem{TiltingPullback}[VitalLemma1]{Proposition}

\newtheorem{HilbertModuliLemma}[VitalLemma1]{Lemma}

\newtheorem{RationalSurfaceSing}[VitalLemma1]{Corollary}
\newtheorem{StabilityConditionEquiv}[VitalLemma1]{Lemma}
\newtheorem{Progenerators}[VitalLemma1]{Proposition}
\newtheorem{PerverseHoms}[VitalLemma1]{Lemma}

\newtheorem{BijectiveConditions}[VitalLemma1]{Theorem}

\newtheorem{GeneralClosedImmersion}[VitalLemma1]{Corollary}
\newtheorem{AlternateModuli}[VitalLemma1]{Lemma}
\newtheorem{DualPullback}[VitalLemma1]{Lemma}

\newtheorem{SekiyaYamauraCorrection}{Proposition}[section]

\newtheorem{StableLemma}[VitalLemma1]{Lemma}
\newtheorem{EquivalencesLemma}[VitalLemma1]{Lemma}

\theoremstyle{definition}

\newtheorem{QGIT1}[VitalLemma1]{Definitions}
\newtheorem{QGIT2}[VitalLemma1]{Definition}
\newtheorem{QGIT3}[VitalLemma1]{Definitions}
\newtheorem{QGIT4}[VitalLemma1]{Definition}

\newtheorem{QGIT7}[VitalLemma1]{Definition}
\newtheorem{Perverse}[VitalLemma1]{Definition}
\newtheorem{TiltingDef}[VitalLemma1]{Definitions}
\newtheorem{resolution}[VitalLemma1]{Definitions}
\newtheorem{KleinianSingularitiesExample}[VitalLemma1]{Example}
\newtheorem{SurfaceQuotientSingularitiesExample}[VitalLemma1]{Example}
\newtheorem{DetermentalSingularitiesExample}[VitalLemma1]{Example}
\newtheorem{DimensionDefinition}[VitalLemma1]{Definitions}

\newtheorem{TiltingAssumptions}[VitalLemma1]{Assumption}

\newtheorem{BridgelandRemark}[VitalLemma1]{Remark}
\newtheorem{ModuliRemark}[VitalLemma1]{Remark}
\newtheorem{MonomorphismRemark}[VitalLemma1]{Remark}
\newtheorem{TautologicalBundleRemark}[VitalLemma1]{Remark}
\newtheorem{FunctorRemark}[VitalLemma1]{Remark}
\newtheorem*{EquivalenceExample}{Example}

\makeatletter
\@addtoreset{VitalLemma1}{section}
\makeatother

\DeclareMathAlphabet{\mathbbm}{U}{bbm}{m}{n}

\let\oldtocsection=\tocsection
\let\oldtocsubsection=\tocsubsection
\let\oldtocsubsubsection=\tocsubsubsection
\renewcommand{\tocsection}[2]{\hspace{0em}\oldtocsection{#1}{#2}}
\renewcommand{\tocsubsection}[2]{\hspace{1em}\oldtocsubsection{#1}{#2}}
\renewcommand{\tocsubsubsection}[2]{\hspace{2em}\oldtocsubsubsection{#1}{#2}}

\newcommand{\Hom}{\textnormal{Hom}}
\newcommand{\Tot}{\textnormal{Tot}}
\newcommand{\CalHom}{\mathcal{H}\textnormal{om}}
\newcommand{\RCalHom}{\mathcal{RH}\textnormal{om}}
\newcommand{\Ext}{\textnormal{Ext}}

\newcommand{\Tor}{\textnormal{Tor}}
\newcommand{\End}{\textnormal{End}}

\newcommand{\CH}{\textnormal{H}}
\newcommand{\SL}{\textnormal{SL}}
\newcommand{\GL}{\textnormal{GL}}

\newcommand{\Spec}{\textnormal{Spec}}
\newcommand{\Image}{\textnormal{Im}}

\newcommand{\Hilb}{\textnormal{Hilb}}

\newcommand{\MaxSpec}{\textnormal{MaxSpec}}

\newcommand{\Proj}{\textnormal{Proj}}
\newcommand{\Per}{\textnormal{Per}}
\newcommand{\Coh}{\textnormal{Coh}}
\newcommand{\Perf}{\textnormal{Perf}}

\newcommand{\Rep}{\textnormal{Rep}}
\newcommand{\op}{\mathrm{op}}
\newcommand{\E}{\mathcal{E}}
\def\mod{\mathop{\textnormal{mod}}\nolimits}
\newcommand{\id}{\mathrm{id}}

\def\rk{\mathop{\rm rk}\nolimits}

\begin{document}

\maketitle
\begin{abstract}

In the setting of a variety $X$ admitting a tilting bundle $T$ we consider the problem of constructing $X$ as a quiver GIT quotient of the algebra $A:=\End_{X}(T)^{\op}$. We prove that if the tilting equivalence restricts to a bijection between the skyscraper sheaves of $X$ and the closed points of a quiver representation moduli functor for $A=\End_{X}(T)^{\op}$ then $X$ is indeed a fine moduli space for this moduli functor, and we prove this result without any assumptions on the singularities of $X$.

As an application we consider varieties which are projective over an affine base such that the fibres are of dimension ≤1, and the derived pushforward of the structure sheaf on X is the structure sheaf on the base. In this situation there is a particular tilting bundle on X constructed by Van den Bergh, and our result allows us to reconstruct X as a quiver GIT quotient for an easy to describe stability condition and dimension vector.  This result applies to flips and flops in the minimal model program, and in the situation of flops shows that both a variety and its flop appear as moduli spaces for algebras produced from different tilting bundles on the variety.

We also give an application to rational surface singularities, showing that their minimal resolutions can always be constructed as quiver GIT quotients for specific dimension vectors and stability conditions. This gives a construction of minimal resolutions as moduli spaces for all rational surface singularities, generalising the G-Hilbert scheme moduli space construction which exists only for quotient singularities.

\end{abstract}

\section{Introduction}
\subsection{Overview}

Any variety $X$ equipped with a tilting bundle $T$ induces a derived equivalence between the bounded derived category of coherent sheaves on $X$ and the bounded derived category of finitely generated left modules for the algebra $A:=\End_{X}(T)^{\op}$. This situation  is similar to the case of an affine variety $\Spec(R)$ where we can construct the commutative algebra $R=\End_X(\mathcal{O}_X)^{\op}$ and there is an abelian equivalence between coherent sheaves on $\Spec(R)$ and finitely generated left $R$-modules. However, whereas in the affine case we can recover the variety $\Spec(R)$ from the algebra $R$, it is not so clear how to recover the variety $X$ from the algebra $A$. One possibility is to present $A$ as the path algebra of a quiver with relations, construct the a moduli space of quiver representations for some dimension vector and stability condition, and attempt to relate this moduli space back to $X$.

While this approach may not work in general there are many examples where this is known to be successful, such as del Pezzo surfaces \cite{KingTilt,CrawWinnMoriDreamSpaces}, minimal resolutions of Kleinian singularities  \cite{KronheimerALE,CassensSlodowy, ExceptionalFiberCrawleyBoevey}, and crepant resolutions of Gorenstein quotient singularities in dimension 3 \cite{McKayBKR,CrawIshii}, which lead us to hope it may work in some other interesting settings.

In this paper we will determine conditions for $X$ to be a fine moduli space for the quiver representation moduli functor  $\mathcal{F}_{A}$, (Section \ref{Quiver GIT moduli functors}), and this will allow us to prove that X is a quiver GIT quotient for a specific stability condition and dimension vector in a large class of examples. These examples include applications to the minimal model program and to resolutions of rational surface singularities. 

This problem was also considered by Bergman and Proudfoot, \cite{BergmanProudfootGIT}, who study embeddings of closed points and tangent spaces to show that a smooth variety is a connected component of the quiver GIT quotient for `great' stability condition and dimension vector. However, their approach cannot be extended to singular varieties and it can be difficult to identify which conditions are `great'. The methods developed in this paper have the advantages of applying to singular varieties, such as those occurring in the minimal model program, and allowing us to identify a specific stability condition and dimension vector in applications.

\subsection{Comparing Moduli Functors}
In developing methods to understand quiver representation moduli functors we are inspired by the following result of Sekiya and Yamaura \cite{SekiyaYamaura}.

\begin{SekiyaYamaura}[{\cite[Theorem 4.20]{SekiyaYamaura}}] Let $B$ be an algebra with tilting module $T$. Define $A=\End_{B}(T)^{\op}$, suppose that both $A$ and $B$ are presented as  path algebras of quivers with relations, and let $\mathcal{F}_{A}$ and $\mathcal{F}_{B}$ denote quiver representation moduli functors on $A$ and $B$ for some choice of stability conditions and dimension vectors. Then if the tilting equivalences 
\begin{center}
\[
\begin{tikzpicture} [bend angle=15, looseness=1]
\node (C1) at (0,0)  {$D^b(B$-$\mod)$};
\node (C2) at (4,0)  {$D^b(A$-$\mod)$};
\draw [->,bend left] (C1) to node[above]  {$\scriptstyle{\mathbb{R}\Hom_{B}(T,-)}$} (C2);
\draw [->,bend left] (C2) to node[below]  {\scriptsize{$ T\otimes ^{\mathbb{L}}_{A}(-)$}} (C1);
\end{tikzpicture}
\]
\end{center}
 restrict to a bijection between $\mathcal{F}_{B}(\mathbb{C})$ and $\mathcal{F}_{A}(\mathbb{C})$ then $\mathcal{F}_{B}$ is naturally isomorphic to $\mathcal{F}_{A}$. 
\end{SekiyaYamaura}

This leads us to the idea of working with a moduli functor for which $X$ is a fine moduli space instead of working with $X$ itself, and we then prove the following variant of Sekiya and Yamaura's result.

\begin{pBigTheorem3}[Theorem \ref{BigTheorem3}]
Let $\pi: X \rightarrow \Spec(R)$ be a projective morphism of varieties. Suppose $X$ is equipped with a tilting bundle $T$, define $A=\End_{X}(T)^{\op}$, and suppose that $A$ is presented as a quiver with relations.  Let $\mathcal{F}_{A}$ be a quiver representation moduli functor on $A$ for some dimension vector and stability condition. Then if the tilting equivalences
\begin{center}
\[
\begin{tikzpicture} [bend angle=20, looseness=1]
\node (C1) at (0,0)  {$D^b(\Coh \, X)$};
\node (C2) at (4,0)  {$D^b( A$-$\mod)$};
\draw [->,bend left] (C1) to node[above]  {\scriptsize{$\mathbb{R}\Hom_{X}(T,-)$}} (C2);
\draw [->,bend left] (C2) to node[below]  {\scriptsize{$ T\otimes ^{\mathbb{L}}_{A}(-)$}} (C1);
\end{tikzpicture}
\]
\end{center} restrict to a bijection between $\mathcal{F}_X(\mathbb{C})$ and $\mathcal{F}_A(\mathbb{C})$ then $\mathcal{F}_{X}$ is naturally isomorphic to $\mathcal{F}_{A}$.

\end{pBigTheorem3}

We recall the definitions of the moduli functors $\mathcal{F}_A$ and $\mathcal{F}_X$ in Sections \ref{Quiver GIT moduli functors} and \ref{Geometric moduli functors}, and note in Appendix \ref{QuiverModuliForAlgebra} that \cite[Theorem 4.20]{SekiyaYamaura} should be stated for the functor defined in Section \ref{Quiver GIT moduli functors} rather than the functor originally defined in \cite[Section 4.2]{SekiyaYamaura}. The moduli functor $\mathcal{F}_X$ is similar to the Hilbert functor of one point on a variety, which is well-known to be represented by $X$, but for lack of a reference in this setting we provide a proof.
\begin{pRepresentability}[{Theorem \ref{Representability}}] Let $\pi:X \rightarrow \Spec(R)$ be a projective morphism of varieties. Then there is an natural isomorphism between the functor of points $\Hom_{\mathfrak{Sch}}(-,X)$ and the moduli functor $\mathcal{F}_{X}$. In particular $X$ is a fine moduli space for $\mathcal{F}_{X}$ with tautological bundle $\Delta_* \mathcal{O}_X$ on $X \times_{\Spec(\mathbb{C})} X$ where $\Delta$ is the diagonal inclusion.
\end{pRepresentability}
Combining these two results we have a method to show when a variety $X$ with tilting bundle $T$ can be recovered via quiver GIT as a fine moduli space for representations of the algebra $A=\End_{X}(T)^{\op}$. 

\begin{pMainTheoremCor} Let $\pi: X \rightarrow \Spec(R)$ be a projective morphism of varieties. Suppose $X$ is equipped with a tilting bundle, $T$,  define $A=\End_{X}(T)^{\op}$, and suppose that $A$ is presented as a quiver with relations. Let $\mathcal{F}_{A}$ be a quiver representation moduli functor on $A$ for some indivisible dimension vector $d$ and generic stability condition $\theta$. Then if the tilting equivalences
\begin{center}
\[
\begin{tikzpicture} [bend angle=20, looseness=1]
\node (C1) at (0,0)  {$D^b(\Coh \, X)$};
\node (C2) at (4,0)  {$D^b( A$-$\mod)$};
\draw [->,bend left] (C1) to node[above]  {\scriptsize{$\mathbb{R}\Hom_{X}(T,-)$}} (C2);
\draw [->,bend left] (C2) to node[below]  {\scriptsize{$ T\otimes ^{\mathbb{L}}_{A}(-)$}} (C1);
\end{tikzpicture}
\]
\end{center} restrict to a bijection between the skyscraper sheaves on $X$ and the $\theta$-stable $A$-modules with dimension vector $d$ then $X$ is a fine moduli space for $\mathcal{F}_A$ and the tautological bundle is the dual of the tilting bundle $T$.
\end{pMainTheoremCor}

\subsection{Applications}
To give an application of this theorem we need a class of varieties with tilting bundles and well-understood tilting equivalences. We consider the situation arising in following theorem of Van den Bergh.
\begin{TiltingBundles1dim}[{\cite[Theorem A]{3Dflops}}] \label{TiltingBundles1dim} Let $\pi:X \rightarrow \Spec(R)$ be a projective morphism of Noetherian schemes such that $\mathbb{R}\pi_* \mathcal{O}_{X} \cong \mathcal{O}_{R}$ and $\pi$ has fibres of dimension $\le 1$. Then there are tilting bundles $T_0$ and $T_1=T_0^{\vee}$ on $X$ such that the derived equivalences $\mathbb{R}\Hom_{X}(T_i,-) : D^b(\Coh \, X) \rightarrow D^b(A_i$-$\mod)$ restrict to equivalences of abelian categories between $^{-i}\Per(X/R)$ and $A_i$-$\mod$, where $A_i=\End_{X}(T_i)^{\op}$.
\end{TiltingBundles1dim}
This gives us a large class of varieties with well-understood tilting equivalences. We recall the definition of $^{-i}\Per(X/R)$ for $i=0,1$  in Definition \ref{Perverse}. We then show that in this situation there is a particular choice of dimension vector $d_{T_0}$ and stability condition $\theta_{T_0}$ such that $X$ occurs as the quiver GIT quotient of $A_0$. 

\begin{pBigTheorem1}[Corollary \ref{BigTheorem1}] Suppose we are in the situation of Theorem \ref{TiltingBundles1dim} and that $X$ and $\Spec(R)$ are both varieties. Then $X$ is the fine moduli space for the quiver representation moduli functor of $A_0=\End_{X}(T_0)^{\op}$ for dimension vector $d_{T_0}$ and stability condition $\theta_{T_0}$.
\end{pBigTheorem1}

See Section \ref{Dimension Vectors and Stability} for the definitions of $\theta_{T_0}$ and $d_{T_0}$. We note they are easy to define and depend only on a decomposition of $T$ into indecomposable summands.

\subsection{Applications to the Minimal Model Program}
The class of varieties in the above corollary includes flips and flops of dimension 3 in the minimal model program. In the setting of smooth, projective 3-folds flops were constructed as components of moduli spaces and shown to be derived equivalent in the work of Bridgeland \cite{FlopsDerivedCategories}, and this work was extended to include projective 3-folds with Gorenstein terminal singularities by Chen \cite{ChenFlops}. These results were reinterpreted more generally via tilting bundles by Van den Bergh \cite{3Dflops}. We can now reinterpret these results once again by combining Corollary \ref{BigTheorem1} with Van den Bergh's results.

It is immediate from Corollary \ref{BigTheorem1} that if $\pi: X\rightarrow \Spec(R)$ is either a flipping or flopping contraction with fibres of dimension $\le 1$ then both $X$ and its flip/flop can be reconstructed as fine moduli spaces with tilting tautological bundles. Further, in the case of flops, the following corollary shows that both $X$ and its flop can be constructed as quiver representation moduli spaces arising from tilting bundles on $X$.

\begin{pFlopGIT}[Corollary \ref{FlopGIT}] Suppose we are in the situation of Corollary \ref{BigTheorem1} and that $\pi:X \rightarrow \Spec(R)$ is a flopping contraction with flop $\pi': X'\rightarrow \Spec(R)$. Then $X$ is the quiver GIT quotient of the algebra $A_0=\End_{X}(T_0)^{\op}$ for dimension vector $d_{T_0}$ and stability condition $\theta_{T_0}$ with tautological bundle $T_0^{\vee}$, and the flop $X'$ is the quiver GIT quotient of the algebra $A_1=\End_{X}(T_1)^{\op}$ for dimension vector $d_{T_1}$ and stability condition $\theta_{T_1}$.
\end{pFlopGIT}

This fits into a general philosophy of having a preferred stability condition defined by a tilting bundle and realising all minimal models via quiver GIT by changing the tilting bundle rather than changing the stability condition. 

\subsection{Applications to Resolutions of Rational Surface Singularities}
Minimal resolutions of  affine rational surface singularities automatically satisfy the conditions of Corollary \ref{BigTheorem1} hence provide another class of examples.
\begin{pRationalSurfaceSing}[Example \ref{RationalSurfaceSing}] Suppose that $X$ is a variety and that $\pi:X \rightarrow \Spec(R)$ is the minimal resolution of a rational surface singularity. Then there is a tilting bundle $T_0$ on $X$ such that $X$ is the fine moduli space of $A_0=\End_{X}(T_0)^{\op}$ for dimension vector $d_{T_0}$ and stability condition $\theta_{T_0}$ with tautological bundle $T_0^{\vee}$. 
\end{pRationalSurfaceSing}
For quotient surface singularities this result was already known when either $G< \SL_2(\mathbb{C})$ \cite{ExceptionalFiberCrawleyBoevey}, or when $G$ was a cyclic or dihedral subgroup of $\GL_{2}(\mathbb{C})$ \cite{CrawSpecialMcKay, RCAA,RCAD1,RCAD2}, but is new in other cases. In particular, for quotient surface singularities the minimal resolution is known to have moduli space interpretation as $G$-$\Hilb(\mathbb{C}^2)$, see \cite{ItoNakamura, HilbertGL2Ishii}, but in general the tautological bundle is not tilting. This corollary extends a similar moduli space interpretation to minimal resolutions of all rational surface singularities such that the tautological bundle is tilting.

\subsection{Outline}

In Section \ref{Preliminaries} we recall a number of preliminary definitions and theorems relating to tilting bundles and quiver GIT which we will need in later sections.  Section \ref{Preliminary Lemmas} consists of a collection of preliminary lemmas which form the bulk of the proofs of our main results. We then prove our main results in Section \ref{Results}, and give an application to a class of examples motivated from the minimal model program, and also to resolutions of rational singularities, in Section \ref{Applications}. Appendix \ref{QuiverModuliForAlgebra} notes and corrects a small error in the results of \cite{SekiyaYamaura}.

\subsection{Acknowledgments}
The author is student at the University of Edinburgh, funded via an Engineering and Physical Sciences Research Council doctoral training grant [grant number EP/J500410/1], and this material will form part of his PhD thesis. The author would like to express his thanks to his supervisors, Dr.\ Michael Wemyss and Prof.\ Iain Gordon, for much guidance and patience, and also to the EPSRC. He would also like to thank an anonymous referee who pointed out an error in an earlier version of the paper, the correction of which led to the discussion in Appendix \ref{QuiverModuliForAlgebra}.

\section{Preliminaries} \label{Preliminaries}

In this section we recall a number of definitions and theorems we will use later, in particular relating to tilting bundles and Quiver GIT. 

\subsection{Geometric and Notational Preliminaries}
We begin by giving some geometric and notational preliminaries. Throughout this paper all schemes will be over $\mathbb{C}$ and a variety will be a scheme which is separated, reduced, irreducible and of finite type over $\mathbb{C}$. In the introduction we stated our results for varieties projective over an affine base, but in fact we will prove our results in the generality of schemes, $X$, arising from projective morphisms $\pi:X \rightarrow \Spec(R)$ of finite type schemes over $\mathbb{C}$.  Such schemes are quasi-projective over $\mathbb{C}$, and hence separated, so are a slight generalisation of varieties projective over an affine base in that they may not be reduced or irreducible. For an affine scheme $\Spec(R)$ we will let $\mathcal{O}_R$ denote $\mathcal{O}_{\Spec(R)}$. We denote the category of coherent sheaves on a scheme $X$ by $\Coh \, X$, we denote the skyscraper sheaf of a closed point $x \in X$ by $\mathcal{O}_x$, and  for a locally free sheaf $\mathcal{F} \in \Coh\, X$ we let $\mathcal{F}^{\vee}$ denote the dual $\CalHom_{X}(\mathcal{F},\mathcal{O}_{X})$. For an algebra $A$ we let $A^{\op}$ denote the opposite algebra of $A$, and $A$-$\mod$ denote the category of finitely generated left $A$-modules.

\subsection{Derived Categories and Tilting} \label{Derived Categories and Tilting}
We recall the definitions of tilting bundles on schemes and several notions related to derived categories that we will make use of later.

Consider a triangulated $\mathbb{C}$-linear category $\mathfrak{C}$ with small direct sums. A subcategory is \emph{localising} if it is triangulated and also closed under all small direct sums. A localising subcategory is necessarily closed under direct summands \cite[Proposition 1.6.8]{Neeman Triangulated Categories}.  An object $T \in \mathfrak{C}$ \emph{generates} if the smallest localising category containing $T$ is $\mathfrak{C}$. 

\begin{TiltingDef} \label{TiltingDef}
Let $\mathfrak{C}$ be a triangulated category closed under small direct sums. An object $T$ in $\mathfrak{C}$ is \emph{tilting} if:
\begin{itemize}
\item[i)] $\Ext^k_{\mathfrak{C}}(T,T)=0 $ for $k \neq 0$.
\item[ii)] $T$ generates $\mathfrak{C}$.
\item[iii)] The functor $\Hom_{\mathfrak{C}}(T,-)$ commutes with small direct sums.
\end{itemize}
\end{TiltingDef}

For $X$ a quasi-projective scheme let $D(X)$ denote the derived category of quasicoherent sheaves on $X$, and $D^b(X)$ denote the bounded derived category of coherent sheaves. For $X$ a Noetherian quasi-projective scheme $D(X)$ is closed under small direct sums \cite[Example 1.3]{Neeman}, and $D(X)$ is compactly generated with compact objects the perfect complexes \cite[Proposition 2.5]{Neeman}. We let $\Perf(X)$ denote the category of perfect complexes on $X$. When $X$ is smooth the category of perfect complexes equals $D^b(X)$.

For an algebra $A$ we let $D(A)$ be the derived category of left modules over $A$, and $D^b(A)$ the bounded derived category of finitely generated left $A$-modules. When $D(X)$ has tilting object a sheaf, $T$,  then define $A:=\End_{X}(T)^{\op}$. When $T$ is a locally free coherent sheaf on $X$ then $T$ is a \emph{tilting bundle} and this gives a derived equivalence between $D(X)$ and $D(A)$.  

\begin{Tilting}[{\cite[Theorem 7.6]{HilleVdB}, \cite[Remark 1.9]{BHTilting}}] \label{BHTilting}
Let $X$ be a scheme that is projective over an affine scheme of finite type,  $\pi:X \rightarrow \Spec(R)$, with tilting bundle $T$ on $X$ and define $A=\End_{X}(T)^{\op}$. Then:
\begin{itemize}
\item[i)] The functor $T_*:= \mathbb{R}\Hom_{X}(T,-)$ is an equivalence between $D(X)$ and $D(A)$. An inverse equivalence is given by the left adjoint $T^*=T\otimes^{\mathbb{L}}_{A}(-)$.
\item[ii)] The functors $T_*,T^*$ remain equivalences when restricted to the bounded derived categories of finitely generated modules and coherent sheaves.
\item[iii)] If $X$ is smooth then $A$ has finite global dimension.
\end{itemize}
Moreover the equivalence $T_*$ is $R$-linear, and $A$ is a finite $R$-algebra. 
\end{Tilting}

\subsection{Quivers and Quiver GIT}  \label{Quiver GIT Definitions}
We set our notation for quivers and then recall the definitions required for quiver geometric invariant theory, following the definitions of King \cite{KingQGIT}.

A \emph{quiver} is a directed multigraph. We will denote a quiver $Q$ by $Q=(Q_0,Q_1)$, with $Q_0$ the set of vertices and $Q_1$ the set of arrows. The set of arrows is equipped with head and tail maps $h,t:Q_1 \rightarrow Q_0$ which take an arrow to the vertices that are its head and tail respectively. We compose arrows from right to left, that is
\begin{equation*}
b.a = \left\{ \begin{array}{c c} ba & \text{ if } h(a)=t(b); \\
0 & \text{ otherwise;} \end{array} \right.
\end{equation*}
and we extend this definition to paths. We recall that there is a trivial path $e_i$ for each vertex $i \in Q_0$ and that these form a set of orthogonal idempotents.

We denote the path algebra by $\mathbb{C}Q$, define $S$ to be the subalgebra of $\mathbb{C}Q$ generated by the trivial paths, and define $V$ to be the $\mathbb{C}$-vector subspace of $\mathbb{C} Q$ spanned by the arrows $a \in Q_1$. Then $S$ is a semisimple $\mathbb{C}$-algebra, $V$ is an $S^e:=S \otimes_{\mathbb{C}} S^{\op}$-module, and $\mathbb{C}Q=T_S(V):= \bigoplus_{i \ge 0 } V^{\otimes_S i}$. Given $\Lambda$ an $S^e$-module we define $I(\Lambda)$ to be the two sided ideal in $\mathbb{C}Q$ generated by $\Lambda$. We then define 
\begin{equation*}
\frac{\mathbb{C} Q}{\Lambda}:=\frac{\mathbb{C} Q}{I(\Lambda)}
\end{equation*}
and refer to it as the path algebra with \emph{relations} $\Lambda$. 

We can now recall the definitions required for quiver GIT.

\begin{QGIT1}
Let $Q=(Q_1,Q_0)$ be a quiver.
 
\begin{enumerate}
\item[i)]
A \emph{dimension vector} for $Q$ is defined to be an element $d \in \mathbb{N}^{Q_0}$ assigning a non-negative integer to each vertex. 

\item[ii)]
A \emph{dimension $d$ representation} of $Q$ is given by assigning to each vertex $i$ the vector space $V_i=\mathbb{C}^{d(i)}$, to each arrow $a$ a linear map $\phi_{a}: V_{t(a)} \rightarrow V_{h(a)}$, and to each trivial path $e_i$ the linear map $id_{V_{i}}$.

\item[iii)]
A \emph{morphism}, $\psi:(V_i,\rho_a) \rightarrow (W_i, \chi_a)$, between two finite dimensional representations is given by a linear map $\psi_i: V_i \rightarrow W_i$ for each vertex $i$ such that for every arrow $a$ we have $ \chi_{a} \circ \psi_{t(a)} = \psi_{h(a)} \circ \rho_a$.

\item[iv)]
The \emph{representation variety}, $\Rep_d(Q)$, is defined to be the set of all representations of $Q$ of dimension $d$, and we note that this is an affine variety.
\end{enumerate} \noindent
We then suppose that the quiver has relations $\Lambda$ defining the algebra $A=\mathbb{C}Q/\Lambda$. 
\begin{enumerate}
\item[v)]
A \emph{representation of the quiver with relations}, $(Q,\Lambda)$, is a representation of $Q$ such that the linear maps assigned to the arrows satisfy the relations among the paths in the quiver. We recall that a representation of a quiver with relations corresponds to a left $\mathbb{C}Q/\Lambda$-module.

\item[vi)]
The \emph{representation scheme} $\Rep_d(Q,\Lambda)$ is the closed subscheme of the affine variety $\Rep_d(Q)$ cut out by the ideal corresponding to the relations $\Lambda$. Closed points of $\Rep_d(Q,\Lambda)$ correspond to representations of  $(Q,\Lambda)$.
\end{enumerate}
\end{QGIT1}

An action of a reductive group on the affine scheme $\Rep_d(Q,\Lambda)$ can now be defined. For $\{\phi_a : a \in Q_1 \}$, a dimension $d$ representation, there is an action of $\GL_{d(i)}(\mathbb{C})$ at vertex $i$ by base change;
\begin{alignat*}{2}
&g.\phi_a =\left\{ \begin{array}{c } g \circ \phi_a   \\
 \phi_a \circ g^{-1}  \\
                                          0  \end{array} \right.
&\begin{array}{c} \text{ if } t(a)=i; \\
 \text{ if } h(a)=i;  \\
\text{ otherwise. } \end{array}
\end{alignat*}
Then  $G:=\GL_d(\mathbb{C}):=\prod_{i \in Q_0} \GL_{d(i)}(\mathbb{C})$ acts on $\Rep_d(Q,\Lambda)$ with kernel $\mathbb{C}^*=\Delta$. We note that orbits of $G$ correspond to isomorphism classes of representations.

\begin{QGIT2}
The \emph{affine quotient} with dimension vector $d$ is defined to be 
\begin{equation*}
\Rep_d(Q,\Lambda) / \kern-0.3em  / G := \Spec(\mathbb{C}[\Rep_d(Q,\Lambda)]^G).
\end{equation*}
\end{QGIT2}

We now recall the definition of stability conditions in order to consider more general GIT quotients of $\Rep_d(Q,\Lambda)$.

\begin{QGIT3}  \hfil
\begin{enumerate} 
\item[i)]
For a dimension vector $d$ a \emph{stability condition} is defined to be a $\theta \in \mathbb{Z}^{Q_0}$ assigning an integer to each vertex of Q such that $\sum_{i \in Q_0} d(i)\theta(i)=0$. For a finite dimensional representation $M$ let $d_M$ be the dimension vector of $M$, and define $\theta(M)= \sum_{i \in Q_0} \theta(i) d_M(i)$. 

\item[ii)]
A finite dimensional representation $M$ is \emph{$\theta$-semistable} if any subrepresentation $N \subset M$ satisfies $\theta(N) \ge \theta(M)$.

\item[iii)]
A $\theta$-semistable representation $M$ is \emph{$\theta$-stable} if there are no non-zero proper subrepresentations $N \subset M$ with $\theta(N)=\theta(M)$.  A stability $\theta$ is \emph{generic} if all $\theta$-semistable representations are stable.

\item[iv)]
For a stability condition $\theta$ define $\Rep_d(Q,\Lambda)^{s}_{\theta}$ to be the set of $\theta$-stable representations, and $\Rep_d(Q,\Lambda)^{ss}_{\theta}$ to be the set of $\theta$-semistable representations. 
\end{enumerate}
\end{QGIT3}

\begin{StableLemma} \label{StableLemma}
Let $d$ be a dimension vector and $\theta$ a stability condition on some quiver with relations.  If $M$ and $N$ are dimension $d$ $\theta$-stable representations then:
\begin{itemize}
\item[i)] If $0 \rightarrow M' \rightarrow M \rightarrow M'' \rightarrow 0$ is a short exact sequence with $M'$ a non-zero and proper submodule of $M$ then $\theta(M')>\theta(M)=0>\theta(M'')$.
\item[ii)] Any non-zero morphism of representations $f:M \rightarrow N$  is an isomorphism.
\item[iii)] Any morphism of representations $f: M \rightarrow M$ is a multiple of the identity.
\end{itemize}
\end{StableLemma}
\begin{proof}
Firstly, if $0 \rightarrow M' \rightarrow M \rightarrow M' \rightarrow 0$ is a short exact sequence and $M'$ is non-zero and proper submodule of $M$ then by the definition of stability $\theta(M')>0=\theta(M)$ and hence $\theta(M'')<0$.

Secondly, suppose $f:M \rightarrow N$ is non-zero and so the kernel is a proper submodule of $M$. If the kernel is trivial than $f$ is an injection and hence an isomorphism as $M$ and $N$ are finite dimensional with the same dimension vector. If the kernel is non-trivial then $\theta(\Image f)<0$ by part i). However, as $\Image f$ is a subrepresentation of $N$ this is a contradiction to the stability of $N$, hence the kernel is trivial and $f$ is an isomorphism.

Finally, if $f:M \rightarrow M$ is a morphism of representations then $f$ defines a morphism of vector spaces $\mathbb{C}^d \rightarrow \mathbb{C}^d$. In particular this map has an eigenvalue $\lambda$ and defines the map of representations $M \xrightarrow{f - \lambda \cdot \id} M$ which is not a surjection. As such it is not an isomorphism and so by part ii) $f=\lambda\cdot  \id$.
\end{proof}

\begin{QGIT4}
Every finite dimensional $\theta$-semistable representation $M$ has a Jordan-Holder filtration
\begin{equation*}
0=M_0 \subset M_1 \subset \dots \subset M_n=M
\end{equation*}
such that each $M_i$ is $\theta$-semistable and each quotient is $\theta$-stable. Two $\theta$-semistable representations are defined to be \emph{S-equivalent} if their Jordan-Holder filtrations have matching composition factors.
\end{QGIT4}
We note that $\theta$-stable objects have length one filtrations hence are $S$-equivalent if and only if they are isomorphic.

Any character of $G$ is given by powers of the determinant character and is of the form
\begin{equation*}
\chi_\theta(g):= \prod_{i \in Q_0} \det(g_i)^ {\theta_i}
\end{equation*}
for some collection of integers $\theta_i$. For a given dimension vector $d$ we will restrict our attention to characters which are trivial on the kernel of the action, $\Delta$, which translates to the condition $\sum \theta(i) d(i)=0$. Hence these characters are in correspondence with stabilities.

We recall that $\Rep_d(Q,\Lambda)$ is affine, and that $f \in \mathbb{C}[\Rep_d(Q,\Lambda)]$ is a \emph{semi-invariant of weight $\chi$} if $f(g.x)=\chi(g)f(x)$ for all $g \in G$ and all $x \in \Rep_d(Q,\Lambda)$. We denote the set of such $f$ as $\mathbb{C}[\Rep_d(Q,\Lambda)]^{G,\chi}$. 

\begin{QGIT7}[\cite{KingQGIT}] \label{QGIT7} The \emph{quiver GIT quotient}, for dimension vector $d$ and stability condition $\theta$, is defined to be the scheme
\begin{equation*}
\mathcal{M}^{ss}_{d,\theta}  := \Proj \left( \bigoplus_{n \ge 0} \mathbb{C}[\Rep_d(Q,\Lambda)]^{G,\chi_{\theta}^n} \right).
\end{equation*}

\end{QGIT7}

It is immediate from this definition that for any stability condition $\theta$ the quiver GIT quotient $\mathcal{M}^{ss}_{d,\theta}$  is projective over the affine quotient $\mathcal{M}^{ss}_{d,0}=\Spec(\mathbb{C}[\Rep_d(Q,\Lambda)]^G)$. 

\subsection{Quivers and Tilting Bundles} \label{Quivers and Tilting bundles} We recall the construction of a quiver with relations from a tilting bundle.

Let $X \rightarrow \Spec(R)$ be a projective morphism of finite type schemes over $\mathbb{C}$. Given a tilting bundle $T'$ on $X$ and a decomposition into  indecomposable summands $T'= \bigoplus_{i=0}^n E_i^{\oplus \alpha_i}$, with $E_i$ and $E_j$ non-isomorphic for $i \neq j$, then $T=\bigoplus_{i=0}^n E_i$ is also a tilting bundle on $X$ and $\End_{X}(T')^{\op}$ is Morita equivalent to $\End_{X}(T)^{\op}$. Hence we will always assume, without loss of generality, that our tilting bundles have a given multiplicity free decomposition into indecomposables, $T=\bigoplus_{i=0}^n E_i$.

We then recall from Theorem \ref{BHTilting} that $A=\End_{X}(T)^{\op}$ is a finite $R$-algebra for $R$ a finite type commutative $\mathbb{C}$-algebra, and we wish to present $A$ as the path algebra of a quiver with relations such that each indecomposable $E_i$ corresponds to the unique idempotent $e_i = id_{E_i}\in \Hom_{X}(E_i,E_i) \subset A=\End_{X}(T)^{\op}$ that is the trivial path at vertex $i$. In particular $1=\sum e_i$ and we have a diagonal inclusion $\bigoplus_{i=0}^n e_i R \subset A$. 

Indeed, we can construct a quiver by creating a vertex $i$  corresponding to each idempotent $e_i$. We then choose a finite set of generators of $e_i A e_j$ as an $R$-module, which is possible as $A$ is finite $R$-module, and create corresponding arrows from vertex $j$ to $i$ for all $0 \le i,j \le n$. We then consider a presentation of $R$ over $\mathbb{C}$ with finitely many generators, possible as it has finite type, and at each vertex add arrows corresponding to each generator of $R$. If we call this quiver $Q$ then by this construction there is a surjection of $R$-algebras $\mathbb{C}Q \rightarrow A$ given by mapping each trivial path to the corresponding idempotent, and each arrow to the corresponding generator. We then take the kernel of this map, $I$, and $\mathbb{C}Q/I \cong A$ as an $R$-algebra. 

We note that this presentation has many unpleasant properties, for example it may be the case that the ideal of relations $I$ is not a subset of the paths of length greater than 1.  In nice situations it is possible to simplify the presentation, see for example the situation considered in \cite[Section 1]{BergmanProudfootGIT}. 

We also note that there is a decomposition of $A$ consider as a left $A$-module into projective modules $A=\bigoplus_{i=0}^n \Hom_{X}(T,E_i)$ where the module $\Hom_{X}(T,E_i)$ corresponds to paths in the quiver starting at vertex $i$.

\subsection{Functor of Points and Moduli Spaces} We recall the definition of the functor of points and the definition of a fine moduli space. Let $\mathfrak{Sch}$ denote the category of finite type schemes over $\mathbb{C}$, let $\mathfrak{Sets}$ denote the category of sets, and let $\mathfrak{R}$ denote the category of finite type commutative $\mathbb{C}$-algebras. Suppose $X \in \mathfrak{Sch}$, then the functor of points for $X$ is defined to be the functor 
\begin{align*}
\Hom_{\mathfrak{Sch}}(-,X): \mathfrak{R} &\rightarrow \mathfrak{Sets} \\
S &\mapsto \Hom_{\mathfrak{Sch}}(\Spec(S),X)
\end{align*}
and by Yoneda's lemma this gives an embedding of $\mathfrak{Sch}$ into the category of functors from $\mathfrak{R}$ to $\mathfrak{Sets}$.  

A functor $\mathcal{F}: \mathfrak{R} \rightarrow \mathfrak{Sets}$ is \emph{representable} if there is some $Y \in \mathfrak{Sch}$ with a natural isomorphism $\nu:\mathcal{F} \rightarrow \Hom_{\mathfrak{Sch}}(-,Y)$. Then $Y$ is said to be a \emph{fine moduli space} for $\mathcal{F}$. 

 A functor $\mathcal{F}$ is said to be  \emph{corepresentable} if there is a natural transformation  $\nu:\mathcal{F} \rightarrow \Hom_{\mathfrak{Sch}}(-,Y)$ such that for any scheme $Y'$ with a natural transformation $\nu':\mathcal{F} \rightarrow \Hom_{\mathfrak{Sch}}(-,Y')$ there is a unique morphism $Y \rightarrow Y'$ factoring $\nu'$ through $\nu$.

Moduli functors and the functor of points could be defined in terms of functors $\mathfrak{Sch}^{\op} \rightarrow \mathfrak{Sets}$ rather than functors $\mathfrak{R} \rightarrow \mathfrak{Sets}$, however as schemes are defined by local affine structure there is a one to one correspondence between contravariant functors from $\mathfrak{Sch}$ to $\mathfrak{Sets}$ and covariant functors from $\mathfrak{R}$ to $\mathfrak{Sets}$ so either viewpoint is equivalent. We choose the one above to automatically simplify later arguments and definitions to considering affine cases. One advantage of the  alternative description is that it is clear to see that if $X$ is a fine moduli space for a moduli functor $\mathcal{F}$ then there is a \emph{tautological element} in $\mathcal{F}(X)$ corresponding to $id \in \Hom_{\mathfrak{Sch}}(X,X)$ under the natural isomorphism.

\subsection{Quiver Representation Moduli Functors} \label{Quiver GIT moduli functors} We recall the definition of a moduli functor for (semi)stable quiver representations. Let $A$ be a $\mathbb{C}$-algebra of finite type. Suppose that $A$ is presented as a quiver with relations and for $B \in \mathfrak{R}$ define $A^B:=A \otimes_{\mathbb{C}} B$. We recall that left $A$-modules correspond to quiver representations. For a dimension vector $d$, stability condition $\theta$, and $B \in \mathfrak{R}$ define the set
\[
    \mathcal{S}_{A,d,\theta}^{(s)s}(B)  := \left. \left\{ M \in  A^B\text{-}\mod \, \begin{array}{| l } \bullet  \text{ $M$ is a finitely generated and flat $B$-module.} \\ \bullet \text{ The $A$-module $  B/\mathfrak{m} \otimes_B M$ has dimension vector $d$} \\
  \text{and is $\theta$-(semi)stable for all maximal ideals $\mathfrak{m}$  of $B$.    } \end{array}  \right\} \right.
\]
and define the \emph{quiver representation moduli functor} to be
\begin{align*}
 \mathcal{F}_{A,d,\theta}^{(s)s}: & \, \mathfrak{R} \rightarrow \mathfrak{Sets} \\
    & B  \mapsto  \left. {\mathcal{S}_{A,d,\theta}^{(s)s}(B) }\middle/ {\sim } \right.
\end{align*} 
where the equivalence $\sim$ is defined by two modules being equivalent if they are isomorphic after tensoring by an invertible $B$-module:  $M \sim N$ if there is a locally free rank one $B$-module $L$ such that $M \otimes_B L \cong N$ as $A^B$ modules. We note that two stable modules are equivalent if and only if they are locally isomorphic.

\begin{EquivalencesLemma}If $M, N \in \mathcal{S}^{s}_{A,d,\theta}(B)$ then $M \sim N$ if and only if $M \otimes_B B_\mathfrak{m} \cong N \otimes_B B_\mathfrak{m}$ for all $\mathfrak{m} \in \MaxSpec(B)$.
\end{EquivalencesLemma}
\begin{proof}
If there exists a rank one locally free $L$ such that $M \otimes_B L \cong N$ then it is clear that $M$ and $N$ are locally isomorphic.  

If $M$ and $N$ are locally isomorphic then consider the $B$-module $L:=\Hom_{A^B}(M,N)$. This is a submodule of $\Hom_{B}(M,N)$ hence is a finitely generated $B$-module as $M$ and $N$ are finitely generated as $B$-modules. For any $\mathfrak{m} \in \MaxSpec(B)$ as $M$ and $N$ are locally free of the same rank $M_{\mathfrak{m}}$ and $N_{\mathfrak{m}}$ are free $B_{\mathfrak{m}}$-modules of the same rank. As such $L_{\mathfrak{m}}:=\Hom_{A^{B_{\mathfrak{m}}}}(M_{\mathfrak{m}},N_{\mathfrak{m}}) $ is free as it is a direct summand of $\Hom_{B_{\mathfrak{m}}}(B_{\mathfrak{m}}^d,B_{\mathfrak{m}}^d) \cong B_{\mathfrak{m}}^{d \times d}$. Further, as $M$ and $N$ are locally isomorphic, $M_{\mathfrak{m}}$ and $N_{\mathfrak{m}}$ are isomorphic as $A^{B_{\mathfrak{m}}}$-modules and hence $L_{\mathfrak{m}} \cong \Hom_{A^{B_{\mathfrak{m}}}}(M_{\mathfrak{m}},M_{\mathfrak{m}})$.  Consider the injection of $B_{\mathfrak{m}}$-modules
\begin{align*}
\phi: B_{\mathfrak{m}}& \rightarrow  \Hom_{A^{B_\mathfrak{m}}}(M_{\mathfrak{m}},M_{\mathfrak{m}})\cong L_{\mathfrak{m}} \\
b & \mapsto b \cdot \id.
\end{align*}
It follows that $L_{\mathfrak{m}}/{\mathfrak{m}}L_{\mathfrak{m}} \cong \Hom_{A^{B_{\mathfrak{m}}}}(M_{\mathfrak{m}}/{\mathfrak{m}}M_{\mathfrak{m}},M_{\mathfrak{m}}/{\mathfrak{m}}M_{\mathfrak{m}})$ as $L_\mathfrak{m}$ and $M_\mathfrak{m}$ are free $B_\mathfrak{m}$ modules, and as $M_{\mathfrak{m}}/{\mathfrak{m}}M_{\mathfrak{m}}$ is a $\theta$-stable $A$-module the map $\phi$ restricted to a fibre $ B_{\mathfrak{m}}/{\mathfrak{m}} B_{\mathfrak{m}} \rightarrow L_{\mathfrak{m}}/{\mathfrak{m}} L_{\mathfrak{m}}$ is surjective by Lemma \ref{StableLemma} iii). Hence by Nakayama's Lemma $\phi$ is surjective, and so it is an isomorphism. Therefore $L_{\mathfrak{m}} \cong B_{\mathfrak{m}}$ and $L$ is a locally free $B$-module of rank 1.  Then the natural map $M \otimes_B \Hom_{A^B}(M,N) \rightarrow N$ can be seen to be an isomorphism by considering localisations at all maximal ideals ${\mathfrak{m}} \in \MaxSpec(B)$ where it reduces to the composition of isomorphisms $L_\mathfrak{m} \cong B_\mathfrak{m}$, $M_\mathfrak{m} \otimes_{B_\mathfrak{m}} B_{\mathfrak{m}} \rightarrow M_\mathfrak{m}$, and $M_\mathfrak{m} \cong N_\mathfrak{m}$.
\end{proof}

By the results of King in \cite{KingQGIT} the quiver representation moduli functor is corepresentable.

\begin{ssModuliGIT}[{\cite[Proposition 5.2]{KingQGIT}}] 
 The scheme $\mathcal{M}^{ss}_{d,\theta}$ corepresents the functor $\mathcal{F}_{A,d,\theta}^{ss}$. In particular, closed points of $\mathcal{M}^{ss}_{d,\theta}$ correspond to $S$-equivalence classes of dimension $d$, $\theta$-semistable $A$-modules.
\end{ssModuliGIT}

 If we restrict to stable representations then the functor is representable and has a fine moduli space.

\begin{sModuliGIT}[{\cite[Proposition 5.3]{KingQGIT}}] \label{sModuliGIT}
Suppose $d$ is indivisible and let $\mathcal{M}^{s}_{d,\theta}$ be the open subscheme of $\mathcal{M}^{ss}_{d,\theta}$ corresponding to the stable points. Then $\mathcal{M}^{s}_{d,\theta}$ is a fine moduli space for $\mathcal{F}_{A,d,\theta}^{s}$.
\end{sModuliGIT}

We note that when $d$ is indivisible and $\theta$ is generic all semistable points are stable and $\mathcal{M}^{ss}_{d,\theta}=\mathcal{M}^{s}_{d,\theta}$ is a fine moduli space. We will later restrict to considering such cases. 

We will often just refer to the functor as $\mathcal{F}_{A}$, recalling the choices of semistablity or stability, $d$, and $\theta$ only when necessary.  We also note that the tautological element for $\mathcal{F}^s_{A,d,\theta}$ is a vector bundle on $\mathcal{M}^s_{d,\theta}$ with each fibre corresponding to a $\theta$-stable representation of $A$ with dimension vector $d$ which we refer to as the \emph{tautological bundle}. 

\begin{FunctorRemark} \label{FunctorRemark} The functor here differs from the functor considered in Sekiya Yamaura, \cite[Definition 4.1]{SekiyaYamaura}, but their results also hold for this functor. See Appendix \ref{QuiverModuliForAlgebra} for more details.  

We also note that the assumption that $A$ is presented as a quiver with relations is not necessary; for any algebra which is finitely generated over a commutative Noetherian ring Van den Bergh defines a functor analogous to $\mathcal{F}^{s}_{A,d,\theta}$  and proves that such a functor is representable when $d$ is indivisible and $\theta$ is generic \cite[Proposition 6.2.1]{NCResolutions}.  We note that local equivalence is used in this setting.
\end{FunctorRemark}

\subsection{Geometric Moduli Functors} \label{Geometric moduli functors}
We define a similar functor for a scheme, $X$, arising in a projective morphism, $\pi: X \rightarrow \Spec(R)$, of finite type schemes over $\mathbb{C}$.

We first introduce several pieces of notation which we will frequently use. Let $\rho:X \rightarrow \Spec(\mathbb{C})$ denote the structure morphism. For $B \in \mathfrak{R}$ we define $X^B:= X \times_{\Spec(\mathbb{C})} \Spec(B)$ and consider the following pullback diagram
\begin{align*} 
\begin{tikzpicture} [bend angle=0]
\node (C1) at (0,0)  {$\Spec(B)$};
\node (C2) at (2,0)  {$\Spec(\mathbb{C})$};
\node (C3) at (0,1.6)  {$X^B$};
\node (C4) at (2,1.6)  {$X$};
\draw [->] (C1) to node[above]  {} (C2);
\draw [->] (C4) to node[right]  {\scriptsize{$\rho$}} (C2);
\draw [->] (C3) to node[left]  {\scriptsize{$\rho^B$}} (C1);
\draw [->] (C3) to node[above]  {\scriptsize{$\rho^X$}} (C4);
\end{tikzpicture}
\end{align*}
which defines the morphisms $\rho^B$ and $\rho^X$ from the structure morphism $\rho:X \rightarrow \Spec(\mathbb{C})$. We note that $X^B$ is also of finite type over $\mathbb{C}$, and has a projective morphism $\pi^B:X^B \rightarrow \Spec(R \otimes_{\mathbb{C}}B)$, see \cite[Remark 1.7]{BHTilting}. Also if $X$ has a tilting bundle $T$ the following result, which is a particular case of the result \cite[Proposition 2.9]{BHTilting} of Buchweitz and Hille, defines a tilting bundle $T^B$ on $X^B$.

\begin{TiltingPullback}[{\cite[Proposition 2.9]{BHTilting}}]
If $T$ is a tilting bundle on $X$ and $A=\End_{X}(T)^{\op}$ then $T^B:=\mathbb{L}\rho^{X*}T$ is a tilting bundle on $X^B$, and $A^B:=\End_{X^B}(T^B)^{\op}=A \otimes_{\mathbb{C}} B$.
\end{TiltingPullback}

We introduce a further piece of notation. For any $B \in \mathfrak{R}$ we let $\MaxSpec(B)$ denote the closed points of $\Spec(B)$, and any $p \in \MaxSpec(B)$ there is a closed immersion $i_p:\Spec(\mathbb{C}) \rightarrow \Spec(B)$ and a pullback diagram 

\[
\begin{array}{c}  
 \tag{$i_p/j_p$}\label{ClosedBaseChange}
 \begin{tikzpicture} [bend angle=0]
\node (C1) at (0,0)  {$\Spec(\mathbb{C})$};
\node (C2) at (2,0)  {$\Spec(B)$};
\node (C3) at (0,1.6)  {$X$};
\node (C4) at (2,1.6)  {$X^B$};
\draw [right hook->] (C1) to node[above]  {\scriptsize{$i_p$}} (C2);
\draw [->] (C4) to node[right]  {\scriptsize{$\rho^B$}} (C2);
\draw [->] (C3) to node[left]  {\scriptsize{$\rho$}} (C1);
\draw [right hook->] (C3) to node[above]  {\scriptsize{$j_p$}} (C4);
\end{tikzpicture}
\end{array}
\]
which we later refer to as the diagram \eqref{ClosedBaseChange}. 

 We can now define the geometric moduli functor. We define $\mathcal{F}_{X}(\mathbb{C})$ to be the set of skyscraper sheaves of $X$ considered up to isomorphism, and for $B \in \mathfrak{R}$ define the sets
\begin{align*}
   \mathcal{S}_X(B) := \left. \left\{ \E \in D^b(X^B) \,\, \begin{array}{ | l } \bullet  \text{ $\mathbb{L}j_p^*\E \in \mathcal{F}_{X}(\mathbb{C})$ for all $p \in \MaxSpec(B)$.} \\ \bullet\text{ $\mathbb{R}\rho^B_*\RCalHom_{X^B}(\mathbb{L}\rho^{X*}\mathcal{F},\E) \in \Perf  (B)$ for all $\mathcal{F} \in \Perf(X)$.} 
 \end{array}  \right\}\right.
\end{align*} 
and the moduli functor
\begin{align*}
 \mathcal{F}_{X}: & \, \mathfrak{R} \rightarrow \mathfrak{Sets} \\
    & B  \mapsto \left. \mathcal{S}_{X}(B) \middle/ \text{$\sim$} \right.
\end{align*} 
where the equivalence $\sim$ is defined by  $\mathcal{E}_1$ being equivalent to $\mathcal{E}_2$ if there is a line bundle $L$ on $\Spec(B)$ such that $\mathcal{E}_1 \otimes_{X^B} \rho^{B*} L \cong \mathcal{E}_2$. We later prove in Theorem \ref{Representability} that $X$ is a fine moduli space for this functor.

\begin{ModuliRemark} \label{ModuliRemark}
It follows immediately from Lemmas \ref{AlternateModuli} and \ref{VitalLemma1}, which we state below, that if $X$ has a tilting bundle $T$ the set $\mathcal{S}_X(B)$ is equivalent to the set
\begin{align*}
 \left. \left\{ \E \in \Coh(X^B) \,\, \begin{array}{ | l } \bullet \text{ $\E$ is flat as a $B$-module.} \\
\bullet  \text{ $j_p^*\E \in \mathcal{F}_{X}(\mathbb{C})$ for all $p \in \MaxSpec(B)$.} \\ \bullet\text{ $\mathbb{R}\Hom_{X^B}(T^B,\E) \in \Perf  (B)$.} \end{array}  \right\}  \right. .
\end{align*} 
\end{ModuliRemark}

\begin{AlternateModuli} \label{AlternateModuli}
Suppose $X$ has a tilting bundle $T$. Then for $\E \in D^b(X^B)$ the condition 
\begin{align*}
 & \bullet\text{ $\mathbb{R}\rho^B_*\RCalHom_{X^B}(\mathbb{L}\rho^{X*}\mathcal{F},\E) \in \Perf  (B)$ for any $\mathcal{F} \in \Perf(X)$}
\intertext{ is equivalent to the condition}
& \bullet\text{ $\mathbb{R}\Hom_{X^B}(T^B,\E) \in \Perf  (B)$.}
\end{align*}
\end{AlternateModuli}
\begin{proof}
Define $\mathfrak{T}$ to be the subset of $\Perf(X)$ which consists of the objects $\mathcal{G} \in \Perf(X)$ such that $\mathbb{R}\Hom_{X^B}(\mathbb{L}\rho^{B*}\mathcal{G},\E) \in \Perf(B)$. Then $\mathbb{R}\Hom_{X^B}(T^B,\E) \in \Perf(B)$ if and only if $T \in \mathfrak{T}$. By \cite[Lemma 2.2]{NeemanRavenel} as $T$ is a tilting bundle and $\mathfrak{T}$ is closed under shifts, triangles, and direct summands  $\mathfrak{T}$ contains $T$ if and only if $\mathfrak{T}=\Perf(X)$ .
\end{proof}

\begin{VitalLemma1}[{\cite[Lemma 4.3]{BridgelandEquivalences}}] \label{VitalLemma1}
Let $f:X \rightarrow Y$ be a morphism of finite type schemes over $\mathbb{C}$, and for each closed point $y \in Y$ let $j_y$ denote the inclusion of the fibre $f^{-1}(y)$. Suppose $\mathcal{E} \in D^b(X)$ is such that $\mathbb{L}j_y^*\mathcal{E}$ is a sheaf for all $y$. Then $\mathcal{E}$ is a coherent sheaf on $X$ which is flat over $Y$.
\end{VitalLemma1}

\begin{BridgelandRemark} In the definition of the moduli functor $\mathcal{F}_X$ we could change the set $\mathcal{F}_X(\mathbb{C})$ of skyscraper sheaves up to isomorphism to, for example, the set of perverse point sheaves as defined by Bridgeland, \cite[Section 3]{FlopsDerivedCategories}, to obtain a functor mirroring Bridgeland's perverse point sheaf moduli functor. Indeed, the results of Section \ref{Preliminary Lemmas} and Theorem \ref{BigTheorem3} do not rely on the fact that $\mathcal{F}_X(\mathbb{C})$ consists of skyscraper sheaves up to isomorphism, but Theorem \ref{Representability} and our applications in Section \ref{Applications} do.
\end{BridgelandRemark}

\section{Preliminary Lemmas} \label{Preliminary Lemmas}
In this section we give a series of lemmas required to prove the main results in the next section. 

\subsection{Derived Base Change}
We first recall the following property, which we will make use of several times.

\begin{DualPullback} \label{DualPullback}
Let $f:X \rightarrow Y$ be a quasi-compact, separated morphism of Noetherian schemes over $\mathbb{C}$. Then if $T \in \Perf(Y)$
\begin{equation*}
\mathbb{L}f^* \RCalHom_{Y}(T,\E) \cong \RCalHom_{X}(\mathbb{L}f^* T,\mathbb{L}f^* \E).
\end{equation*} 
for any $\E \in D^b(Y)$.
\end{DualPullback}
\begin{proof}

We consider the two functors \begin{align*}
\Hom_{D^b(X)}(\mathbb{L}f^* \RCalHom_{Y}(T,\E),-)&: D^b(X) \rightarrow \mathfrak{Sets}, \text{ and} \\
\Hom_{D^b(X)}( \RCalHom_{X}(\mathbb{L}f^*T,\mathbb{L}f^*\E),-)&: D^b(X) \rightarrow \mathfrak{Sets}.
\end{align*}
 We will show these are naturally isomorphic, and it then follows that $\mathbb{L}f^* \RCalHom_{Y}(T,\E) \cong \RCalHom_{X}(\mathbb{L}f^*T,\mathbb{L}f^*\E)$ as they represent the same functor under the Yoneda embedding. This follows from the chain of natural isomorphisms
\begin{align*}
\Hom_{D(X)}(\mathbb{L}f^*\RCalHom_{Y}(T,\E), -)  & \cong \Hom_{D(Y)}(\RCalHom_{Y}(T,\E) , \mathbb{R} f_*(-)) \tag{adjunction}\\
&\cong \Hom_{D(Y)}(\E,  T \otimes^{\mathbb{L}}_{Y} \mathbb{R}f_*(-)) \tag{$T$ perfect}\\
& \cong  \Hom_{D(Y)}(\E,  \mathbb{R}f_* ( \mathbb{L}f^* T \otimes^{\mathbb{L}}_{X} (-))) \tag{projection} \\
& \cong  \Hom_{D(X)}(\mathbb{L}f^* \E,  \mathbb{L}f^* T \otimes^{\mathbb{L}}_{Y} (-))  \tag{adjunction}\\
& \cong \Hom_{D(X)} (\RCalHom_{X}(\mathbb{L}f^*T,\mathbb{L}f^* \E),-).  \tag{$\mathbb{L}f^*T$ perfect}
\end{align*}
\end{proof}

We then recall the following derived base change results.
\begin{DerivedBaseChange}\label{Derived Base Change}  Let $\pi:X\rightarrow \Spec(R)$ be a projective morphism of finite type schemes over $\mathbb{C}$, and let $B,C \in \mathfrak{R}$. Consider the following pullback diagram for a morphism $u:\Spec(B) \rightarrow \Spec(C)$, where we use the notation of Section \ref{Geometric moduli functors}. 
\begin{align*}
\begin{tikzpicture} [bend angle=50]
\node (C1) at (0,0)  {$\Spec(B)$};
\node (C2) at (2,0)  {$\Spec(C)$};
\node (C3) at (0,1.6)  {$X^B$};
\node (C4) at (2,1.6)  {$X^C$};
\draw [->] (C4) to node[right]  {\scriptsize{$\rho^C$}} (C2);
\draw [->] (C3) to node[left]  {\scriptsize{$\rho^B$}} (C1);
\draw [->] (C3) to node[above]  {\scriptsize{$v$}} (C4);
\draw [->] (C1) to node[above]  {\scriptsize{$u$}} (C2);
\end{tikzpicture}
\end{align*} Suppose $\E \in D^b(X^C)$. Then
 \begin{equation*} \begin{aligned}[t] \mathbb{L}u^* \mathbb{R} \rho^C_* \E \cong \mathbb{R} \rho^B_* \mathbb{L}v^*  \E. \end{aligned} \end{equation*}

Suppose further that $X$ has a tilting bundle $T$ and define $A=\End_X(T)^{\op}$. If the $A^C$-module $\mathbb{R}\Hom_{X^C}(T^C,\E)$ is flat as a $C$-module then
\begin{equation*}
B \otimes_C \mathbb{R}\Hom_{X^C}(T^C,\E) \cong \mathbb{R}\Hom_{X^B}(T^B,\mathbb{L}v^* \E)
\end{equation*} as $A^B$-modules.

Also, if $L$ is a line bundle on $\Spec(B)$ then \[ \mathbb{R}\Hom_{X^B}(T^B,\E \otimes_{X^B} \rho^{B*} L) \cong \mathbb{R}\Hom_{X^B}(T^B,\E) \otimes_B L\]
as $A^B$-modules.
\end{DerivedBaseChange}
\begin{proof}

As $X^C$ is flat over $\Spec(C)$, for any $x \in X^C$ and any $b  \in \Spec(B)$ such that $\rho^C(x)=u(b)=c$  we have that $\Tor_i^{\mathcal{O}_{C,c}}(\mathcal{O}_{B,b},\mathcal{O}_{X^C,x})=0$ for all $i\neq 0$. Hence $X^C$ and $\Spec(B)$ are Tor independent over $\Spec(C)$, and so the first result follows from \cite[Lemma 35.16.3(Tag 08IB)]{StacksProject}.

The second result follows by applying the first result and the previous lemma: 
\begin{align*}
B \otimes_C \mathbb{R}\Hom_{X^C}(T^C,\E)& \cong  \mathbb{L}u^* \mathbb{R}\rho^C_* \RCalHom_{X^C}(T^C,\E) \\
& \cong \mathbb{R}\rho^B_* \mathbb{L}v^* \RCalHom_{X^C}(T^C,\E)  \tag{$\mathbb{L}u^* \mathbb{R} \rho^C_* \cong \mathbb{R} \rho^B_* \mathbb{L}v^* $}\\
& \cong \mathbb{R}\rho^B_* \RCalHom_{X^B}(\mathbb{L}v^*T^C,\mathbb{L}v^* \E) \tag{Lemma \ref{DualPullback}} \\
& \cong \mathbb{R}\Hom_{X^B}(T^B,\mathbb{L}v^*\E).
\end{align*}

The final assertion follows by the projection formula (\cite[Lemma 20.41.3 (Tag 0B54)]{StacksProject}:
\begin{align*}
 \mathbb{R}\Hom_{X^B}(T^B,\E \otimes_{X^B} \rho^{B*} L)&:= \mathbb{R}\rho^B_* \CalHom_{X^B}(T^ B,\E \otimes_{X^B} \rho^{B*} L) \\
&= \mathbb{R}\rho^B_*((T^{B})^\vee \otimes_{X^B}\E \otimes_{X^B} \rho^{B*} L)  \tag{$T^B$ perfect}\\
&= \mathbb{R}\rho^B_*((T^{B})^\vee \otimes_{X^B}\E) \otimes_{B} L   \tag{projection formula} \\
&= \mathbb{R}\rho^B_* \CalHom_{X^B}(T^ B,\E )\otimes_{B} L  \tag{$T^B$ perfect} \\
&=\mathbb{R}\Hom_{X^B}(T^B,\E) \otimes_{B} L.
\end{align*}
\end{proof}

The following corollary is also useful.
\begin{DerivedBaseChangeCor} \label{DerivedBaseChangeCor}Let $X$ be a scheme of finite type over $\mathbb{C}$, and let $B \in \mathfrak{R}$. Suppose that $\E \in D^b(X^B)$ is such that $\mathbb{R}\rho^B_* \E \in D^b(B)$, and that for any $p \in \MaxSpec(B)$ with diagram \eqref{ClosedBaseChange} we have that $\mathbb{R}\rho_* \mathbb{L} j_p^* \E$ is a coherent sheaf on $\Spec(\mathbb{C})$. Then  $\mathbb{R}\rho^B_*\E$ is a flat coherent sheaf on $\Spec(B)$.
\end{DerivedBaseChangeCor}
\begin{proof}
By assumption $\mathbb{R}\rho^B_* \E \in D^b(\Spec(B))$, hence by Lemma \ref{VitalLemma1} if $\mathbb{L}i_p^* \mathbb{R} \rho^B_* \E $ is a sheaf for all embeddings of closed points, $i_p:\Spec(\mathbb{C}) \rightarrow \Spec(B)$, then $\mathbb{R}\rho^B_*\E$ is a flat coherent sheaf on $\Spec(B)$. For any such $p \in \MaxSpec(B)$ by Theorem \ref{Derived Base Change} $\mathbb{L}i_p^* \mathbb{R} \rho^B_* \E \cong \mathbb{R} \rho_* \mathbb{L}j_p^*  \E$ which is a sheaf  by the hypothesis.
\end{proof}

\subsection{Natural Transformations} \label{Natural Transformation} In this section let $\pi:X\rightarrow \Spec(R)$ be a projective morphism of finite type schemes over $\mathbb{C}$. Suppose that $X$ has a tilting bundle $T$ and that $A=\End_{X}(T)^{\op}$ is presented as a quiver with relations. Choose some dimension vector $d$ and stability condition $\theta$ in order to define $\mathcal{F}_{A}:= \mathcal{F}_{A,d,\theta}^{ss}$.  We aim to define a natural transformation, $\eta$, between the moduli functors $\mathcal{F}_{X}$ and $\mathcal{F}_{A}$ defined in Sections \ref{Geometric moduli functors} and \ref{Quiver GIT moduli functors}. We define $\eta:\mathcal{F}_X \rightarrow \mathcal{F}_A$  by 
\begin{align*}
\eta_B: \mathcal{F}_X(B) &\rightarrow \mathcal{F}_A(B) \\
\E & \mapsto \mathbb{R}\Hom_{X^B}(T^B,\E)
\end{align*}
for any $B \in \mathfrak{R}$, and we must check when this is well defined. 

\begin{VitalLemma2}  \label{VitalLemma2} Suppose $\eta_{\mathbb{C}}$ is well defined. Then $\eta$ is a well defined natural transformation and $\eta_B$ is injective for all $B \in \mathfrak{R}$.
\end{VitalLemma2}
\begin{proof} To prove that $\eta$ is well defined we must check the following for any $ B \in \mathfrak{R}$ and any $\E \in \mathcal{F}_X(B)$:
\begin{enumerate}
\item[i)]
$\mathbb{R}\Hom_{X^B}(T^B,\E)$ is a $B$-module which is flat and finitely generated.

\item[ii)]
 For all maximal ideals $\mathfrak{m}$ of $B$ the $A$-module $B/\mathfrak{m}\, \otimes_{B} \mathbb{R}\Hom_{X^B}(T^B,\E)$ is in $\mathcal{F}_A(\mathbb{C})$.

\item[iii)]
If $\E_1$ and $\E_2$ are equivalent in $\mathcal{F}_X(B)$ then $\mathbb{R}\Hom_X(T,\E_1)$ and $\mathbb{R}\Hom_X(T,\E_2)$ are equivalent in $\mathcal{F}_A(B)$. 
\end{enumerate}

Firstly we check i). It follows from the definition of $\mathcal{F}_{X}(B)$ that $\mathbb{R}\Hom_{X^B}(T^B,\E) \in \Perf(B) \subset D^b(B)$. Then by Lemma \ref{DerivedBaseChangeCor} if $
  \mathbb{R}\rho_* \mathbb{L}j_p^* \RCalHom_{X^B}(T^B,\E)$ is a sheaf on $\Spec(\mathbb{C})$ for all $p \in \MaxSpec(B)$ with diagrams \eqref{ClosedBaseChange} then \[  \mathbb{R}\rho^B_* \RCalHom_{X^B}(T^B,\E) \cong \mathbb{R}\Hom_{X^B}(T^B,\E) \] is a flat and finitely generated $B$-module.  For all $p \in \MaxSpec(B)$ with diagrams \eqref{ClosedBaseChange}
\begin{equation*}
  \mathbb{R}\rho_* \mathbb{L}j_p^* \RCalHom_{X^B}(T^B,\E)\cong\mathbb{R} \Hom_X(T,\mathbb{L}j_p^* \E),
\end{equation*}
 by Lemma \ref{DualPullback} and $\mathbb{R} \Hom_X(T,\mathbb{L}j_p^* \E) \in \mathcal{F}_A(\mathbb{C})$ as $\mathbb{L}j_p^*\E \in \mathcal{F}_X(\mathbb{C})$ by the definition of $\mathcal{F}_X(B)$ and as $\eta_{\mathbb{C}}$ is well defined. Hence    $\mathbb{R}\rho_* \mathbb{L}j_p^* \RCalHom_{X^B}(T^B,\E)\cong\mathbb{R} \Hom_X(T,\mathbb{L}j_p^* \E)$  is a coherent sheaf on $\Spec(\mathbb{C})$, so we have proved i).

Secondly, to prove ii), we note for any maximal ideal $\mathfrak{m}$ of $B$ there is a corresponding closed point $p \in \MaxSpec(B)$ and diagram \eqref{ClosedBaseChange}.  Then if $\mathcal{E} \in \mathcal{F}_{X}(B)$ for each maximal ideal we have $B/\mathfrak{m} \otimes_B \mathbb{R}\Hom_{X^B}(T^B,\mathcal{E})  \cong \mathbb{R}\Hom_{X}(T,\mathbb{L}j_p^*\mathcal{E})$ by Lemma \ref{Derived Base Change} as $\mathbb{R}\Hom_{X^B}(T^B,\mathcal{E})$ is a flat $B$ module. Hence $B/\mathfrak{m} \otimes_B \mathbb{R}\Hom_{X^B}(T^B,\mathcal{E})  \in \mathcal{F}_{A}(\mathbb{C})$  as $\eta_{\mathbb{C}}$ is well defined and $\mathbb{L}j_p^*\mathcal{E} \in \mathcal{F}_{X}(\mathbb{C})$ by the definition of $\mathcal{F}_X(B)$. 

To prove part iii) let  $\E_1$ and $\E_2$ be equivalent elements of $\mathcal{F}_{X}(B)$. Then there exists some line bundle $L$ on $\Spec(B)$ such that $\E_1 \otimes_{X^B} \rho^{B*} L \cong \E_2$, and so by Lemma \ref{Derived Base Change}
\begin{align*} \mathbb{R}\Hom_{X^B}(T^B,\E_2) & \cong \mathbb{R}\Hom_{X^B}(T^B,\E_1\otimes_{X^B} \rho^{B*} L ) \\
& \cong  \mathbb{R}\Hom_{X^B}(T^B,\E_1 )\otimes_{B} L.
\end{align*}
This shows that $\eta_B(\E_1)$ and $\eta_B(\E_2)$ are equivalent in $\mathcal{F}_X(B)$ and proves part iii).

We now show that $\eta$ is a natural transformation. Suppose that $B,C \in \mathfrak{R}$ and $u:\Spec(B)\rightarrow \Spec(C)$, then we have the base change diagram
\begin{align*}
\begin{tikzpicture} [bend angle=0]
\node (C3) at (0,1.6)  {$\Spec(B)$};
\node (C4) at (2,1.6)  {$\Spec(C)$};
\node (C5) at (0,3.2)  {$X^B$};
\node (C6) at (2,3.2)  {$X^C$};
\draw [->] (C3) to node[above]  {\scriptsize{$u$}} (C4);
\draw [->] (C5) to node[above]  {\scriptsize{$v$}} (C6);
\draw [->] (C6) to node[right]  {\scriptsize{$\rho^C$}} (C4);
\draw [->] (C5) to node[left]  {\scriptsize{$\rho^B$}} (C3);
\end{tikzpicture}
\end{align*}
and we consider the diagram 
\begin{align*}
\begin{tikzpicture} [bend angle=0]
\node (C3) at (0,1.6)  {$\mathcal{F}_{X}(B)$};
\node (C4) at (5,1.6)  {$\mathcal{F}_{A}(B)$};
\node (C5) at (0,3.2)  {$\mathcal{F}_{X}(C)$};
\node (C6) at (5,3.2)  {$\mathcal{F}_{A}(C)$};
\draw [->] (C3) to node[above]  {\scriptsize{$\mathbb{R}\Hom_{X^B}(T^B,-)$}} (C4);
\draw [->] (C5) to node[above]  {\scriptsize{$\mathbb{R}\Hom_{X^C}(T^C,-)$}} (C6);
\draw [->] (C6) to node[right]  {\scriptsize{$B \otimes_{C} (-)$}} (C4);
\draw [->] (C5) to node[left]  {\scriptsize{$\mathbb{L}v^*$}} (C3);
\end{tikzpicture}
\end{align*}
and to show that $\eta$ is natural we must check that this commutes. For $\E \in \mathcal{F}_{X}(C)$ as $  \mathbb{R}\Hom_{X^C}(T^C,\E)$ is a flat $C$-module
\begin{align*}
B \otimes_C \mathbb{R}\Hom_{X^C}(T^C,\E) \cong  \mathbb{R}\Hom_{X^B}(T^B,\mathbb{L}v^*\E)
\end{align*}
as $A^B$-modules by Lemma \ref{Derived Base Change}. Hence $\eta$ is natural. 

We now show that $\eta_B$ is injective.  Suppose that $\E_1, \E_2 \in \mathcal{F}_{X}(B)$ and $\mathbb{R}\Hom_{X^B}(T^B,\mathcal{E}_1)$ is equivalent to $\mathbb{R}\Hom_{X^B}(T^B,\mathcal{E}_2)$, hence there exists an invertible $B$-module $L$ such that $\mathbb{R}\Hom_{X^B}(T^B,\mathcal{E}_1) \otimes_B L \cong \mathbb{R}\Hom_{X^B}(T^B,\mathcal{E}_2)$. By Lemma \ref{Derived Base Change} 
\begin{align*}
  \mathbb{R}\Hom_{X^B}(T^B,\mathcal{E}_2) & \cong \mathbb{R}\Hom_{X^B}(T^B,\mathcal{E}_1) \otimes_B L \\
& \cong \mathbb{R}\Hom_{X^B}(T^B,\mathcal{E}_1 \otimes_{X^B} \rho^{B*} L)
\end{align*} and hence $\mathcal{E}_1 \otimes_{X^B} \rho^{B*} L \cong \mathcal{E}_2$ as $\mathbb{R}\Hom_{X^B}(T^B,-)$ is an equivalence of derived categories. Hence $\eta_B$ is injective.
\end{proof}

\begin{VitalLemma3} \label{VitalLemma3} With the assumptions as in Lemma \ref{VitalLemma2}, if $\eta_\mathbb{C}$ is also  bijective with inverse $T \otimes_A^{\mathbb{L}}(-)$ then $\eta_B$ is bijective for all $B\in \mathfrak{R}$.
\end{VitalLemma3}

\begin{proof}
We suppose that $\eta_{\mathbb{C}}$ is bijective with inverse $T \otimes_A^{\mathbb{L}}(-)$ and we show that $\eta_B$ is surjective. We consider $M \in \mathcal{F}_{A}(B)$ and note that as $T^B$ is a tilting bundle there exists some $\mathcal{E} \in D^b(X^B)$ such that $\mathbb{R}\Hom_{X^B}(T^B,\mathcal{E}) \cong M$. Then if we can show that $\E \in \mathcal{F}_X(B)$ then we have proved that $\eta_B$ is surjective. We check first that $\mathbb{L}j_p^*\E \in \mathcal{F}_X(\mathbb{C})$ for any $p \in \MaxSpec(B)$ and diagram \eqref{ClosedBaseChange}, and then we check that $\mathbb{R}\Hom_{X^B}(\mathbb{L}\rho^{B*}\mathcal{G},\E) \in \Perf(B)$ for any $\mathcal{G} \in \Perf(X)$. 

Firstly, for any maximal ideal $\mathfrak{m}$ of $B$ there is a corresponding closed point $p \in \MaxSpec(B)$ and diagram \eqref{ClosedBaseChange}, and by Lemma \ref{Derived Base Change}
\begin{align*}
B/\frak{m} \otimes_B M \cong  \mathbb{R}\Hom_{X}(T,\mathbb{L}j_p^*\mathcal{E}) 
\end{align*}
as $M$ is flat over $B$. As $B/\frak{m} \otimes_B M \cong \mathbb{R}\Hom_{X}(T,\mathbb{L}j_p^*\E) \in \mathcal{F}_{A}(\mathbb{C})$ and $\eta_{\mathbb{C}}$ is bijective with inverse $T \otimes_{A}^{\mathbb{L}}(-)$ it follows that $\mathbb{L}j_p^*\mathcal{E} \cong T  \otimes_{A}^{\mathbb{L}}\mathbb{R}\Hom_{X}(T,\mathbb{L}j_p^*\E) \in \mathcal{F}_{X}(\mathbb{C})$. 

As $M$ is a flat and finitely generated $B$-module $\mathbb{R}\Hom_{X^B}(T^B,\E)\cong M \in \Perf(B)$ so the second condition holds by Lemma \ref{AlternateModuli}. Hence $\mathcal{E} \in \mathcal{F}_{X}(B)$ and $\eta_B$ is surjective. 
\end{proof}

\section{Results} \label{Results}

In this section we state our main result, which follows from the previous lemmas, and we also show that the moduli functor $\mathcal{F}_{X}$ is represented by $X$. We will find several applications of these results in the next section.

\begin{BigTheorem3} \label{BigTheorem3}

Let $\pi: X \rightarrow \Spec(R)$ be a projective morphism of finite type schemes over $\mathbb{C}$. Suppose $X$ is equipped with a tilting bundle $T$, define $A=\End_{X}(T)^{\op}$, and suppose that $A$ is presented as a quiver with relations. If there exists a dimension vector $d$ and stability condition $\theta$ defining the moduli functor $\mathcal{F}_{A}:= \mathcal{F}^{ss}_{A,\theta,d}$ such that the tilting equivalence
\begin{center}
\[
\begin{tikzpicture} [bend angle=20, looseness=1]
\node (C1) at (0,0)  {$D^b(X)$};
\node (C2) at (4,0)  {$D^b(A)$};
\draw [->,bend left] (C1) to node[above]  {\scriptsize{$\mathbb{R}\Hom_{X}(T,-)$}} (C2);
\draw [->,bend left] (C2) to node[below]  {\scriptsize{$ T \otimes ^{\mathbb{L}}_{A} (-)$}} (C1);
\end{tikzpicture}
\]
\end{center} restricts to a bijection between $\mathcal{F}_{X}(\mathbb{C})$ and $\mathcal{F}_{A}(\mathbb{C})$ then    the map $\eta:\mathcal{F}_{X} \rightarrow \mathcal{F}_{A}$  defined by $\eta_B: \E \mapsto \mathbb{R}\Hom_{X^B}(T^B,\E)$ is a natural isomorphism.
\end{BigTheorem3}

\begin{proof}
This follows from Lemmas \ref{VitalLemma2} and \ref{VitalLemma3}.
\end{proof}

We now prove that the moduli functor $\mathcal{F}_{X}$ has $X$ as a fine moduli space. This closely follows the proof of the more general result \cite[Theorem 2.10]{Calabrese2013} of Calabrese and Groechenig, which we split into the following lemma and theorem in our setting.

\begin{HilbertModuliLemma}\label{HilbertModuliLemma} Let $\pi:X \rightarrow \Spec(R)$ be a projective morphism of finite type schemes over $\mathbb{C}$. Suppose that $B \in \mathfrak{R}$ and that $\E \in \mathcal{F}_{X}(B)$. Then:
\begin{itemize}
\item[i)] The element $\E$ is a coherent sheaf on $X^B$ that is flat over $\Spec(B)$ and $\mathbb{R}\rho^B_* \E$ is a line bundle on $\Spec(B)$.
\end{itemize}
Let $\iota:Z \rightarrow X^B$ be the schematic support of $\E$. Then:
\begin{itemize}
\item[ii)]  The morphism $\rho^B \circ \iota: Z \rightarrow \Spec(B)$ is an isomorphism.
\item[iii)] There exists a line bundle $L$ on $\Spec(B)$ such that $\E\cong \iota_* \mathcal{O}_Z  \otimes_{X^B} \rho^{B*} L $.  
\end{itemize}
\end{HilbertModuliLemma}
\begin{proof} Firstly, as $\E \in \mathcal{F}_X(B)$ it is a coherent sheaf on $X^B$ which is flat over $\Spec(B)$ by Remark \ref{ModuliRemark}. Then $\mathcal{O}_X \in \Perf(X)$ hence by the definition of $\mathcal{F}_X(B)$ we know that $\mathbb{R}\rho^B_*\E = \mathbb{R}\Hom_{X^B}(\mathcal{O}_{X^B},\E) \in \Perf(B) \subset D^b(B)$. It follows that $\mathbb{R}\rho^B_* \E$  is a flat coherent sheaf on $\Spec(B)$ by Corollary \ref{DerivedBaseChangeCor} as for all $p \in \MaxSpec(B)$ with diagrams \eqref{ClosedBaseChange} $\mathbb{R}\rho_* \mathbb{L}j_p^* \E = \mathbb{C}$  as  $\mathbb{L}j_p^* \E$ is a skyscraper sheaf. As $\mathbb{L}i_p^*\mathbb{R}\rho^B_* \E= \mathbb{C}$ the flat coherent sheaf $\mathbb{R}\rho^B_* \E$ has rank 1 and is a line bundle on $\Spec(B)$.

To prove ii)  let $Z$ denote the schematic support of $\E$ with closed immersion $\iota:Z \rightarrow X^B$ and let $\mathcal{G}:=\iota^*\E$ denote the sheaf on $Z$ such that $\iota_* \mathcal{G} \cong \E$. We then have the diagram

\begin{align*}
\begin{tikzpicture} [bend angle=0]
\node (C2) at (2,0)  {$\Spec(B)$};
\node (C3) at (0,2)  {$Z$};
\node (C4) at (2,2)  {$X^B$};
\node (B1) at (4,0) {$\Spec(\mathbb{C})$};
\node (B2) at (4,2) {$X$};
\draw [->] (C4) to node[above]  {\scriptsize{$\rho^{X}$}} (B2);
\draw [->] (C4) to node[right]  {\scriptsize{$\rho^B$}} (C2);
\draw [right hook->] (C3) to node[above]  {\scriptsize{$\iota$}} (C4);
\draw [->] (C3) to node[left]  {\scriptsize{$\psi$}} (C2);
\draw [->] (C2) to node[below]  {} (B1);
\draw [->] (B2) to node[right]  {\scriptsize{$\rho$}} (B1);
\end{tikzpicture}
\end{align*}
where we define $\psi= \rho^B \circ \iota$. We recall that $X^B$ is projective over affine and $\rho^B$ can be factored into 
\begin{equation*}
X^B \xrightarrow{\pi^B} \Spec(R \otimes_{\mathbb{C}} B) \xrightarrow{\alpha^B} \Spec(B)
\end{equation*}
where $\pi^B$ is projective and $\alpha^B$ is affine. We then see that as $\iota$ is a closed immersion, hence proper,  $\pi^B \circ \iota$ is a proper map and it has affine fibres, as the fibres are all empty or points, so is an affine morphism by \cite[Theorem 8.5]{Rydh}. We then conclude that $\psi=\alpha^B \circ (\pi^B \circ \iota)$ is an affine morphism as it is the composition of two affine morphisms, in particular $\psi_*$ is exact.

We recall that $\psi_* \mathcal{G}$ is defined as an $\mathcal{O}_{B}$-module via its definition as an $\psi_* \mathcal{O}_Z$-module by the map of rings
\begin{equation*}
\mathcal{O}_B \rightarrow \psi_* \mathcal{O}_Z \rightarrow \End_{\psi_* \mathcal{O}_Z}(\psi_* \mathcal{G}) \rightarrow \End_{\mathcal{O}_B}(\psi_* \mathcal{G}).
\end{equation*}
Then as $\psi_*\mathcal{G} \cong \mathbb{R} \rho^B_* \E$ is a line bundle this series of maps composes to an isomorphism, hence the first map is injective and the last surjective. We also note that the last map is the forgetful map so is also injective, thus is an isomorphism. Hence the middle map is surjective. Then as the support of $\mathcal{G}$ is $Z$ the middle map is also injective, hence is an isomorphism, so in fact the first map must also be an isomorphism. In particular this implies $\mathcal{O}_B  \cong \psi_* \mathcal{O}_Z$ and as $\psi$ is affine it follows that $Z\cong \Spec (B)$ and $\psi$ is an isomorphism. 

To prove iii) define $L=\psi_*(\mathcal{G})$, which is a line bundle by part i). Then
\begin{align*}
\E &\cong \iota_* \mathcal{G} \\
& \cong \iota_*( \mathcal{O}_Z  \otimes_Z \mathcal{G}) \\
& \cong \iota_*( \mathcal{O}_Z \otimes_Z \iota^* \rho^{B*} L)  \tag{ $\psi=\rho^B \circ \iota$ an isomorphism}\\
& \cong \iota_*( \mathcal{O}_Z ) \otimes_{X^B} \rho^{B*} L.  \tag{projection formula}
\end{align*}
\end{proof}
\begin{Representability} \label{Representability}
Let $\pi:X \rightarrow \Spec(R)$ be a projective morphism of finite type schemes over $\mathbb{C}$. Then there is a natural isomorphism between the functor of points $\Hom_{\mathfrak{Sch}}(-,X)$ and the moduli functor $\mathcal{F}_{X}$. In particular $X$ is a fine moduli space for $\mathcal{F}_{X}$ with tautological bundle $\Delta_* \mathcal{O}_X$ on $X \times_{\Spec(\mathbb{C})} X$ where $\Delta: X \rightarrow X \times_{\Spec(C)} X$ is the diagonal map.
\end{Representability}
\begin{proof}
Consider 
\begin{equation*}
\mu: \Hom_{\mathfrak{Sch}}(-,X) \rightarrow \mathcal{F}_{X}
\end{equation*}
defined by 
\begin{align*}
\mu_C: (g:\Spec(C) \rightarrow X) \mapsto ((\Gamma_g)_* \mathcal{O}_{C})
\end{align*}
for $C \in \mathfrak{R}$, where $\Gamma_g: \Spec(C) \rightarrow X^C$ is the graph of $g$. The graph is a closed immersion as $X$ is separated, and hence $\Gamma_g$ is affine and $(\Gamma_{g})_*$ is exact.

We now show this is a well defined natural transformation. To show that it is well defined we consider a morphism $g: \Spec(C) \rightarrow X$ and check that $(\Gamma_{g})_*\mathcal{O}_C \in \mathcal{F}_X(C)$. Firstly, as $\Gamma_g$ is a closed immersion it is proper, hence $(\Gamma_{g})_*\mathcal{O}_C $ is a coherent sheaf \cite[Lemma 29.17.2 (Tag 0205)]{Stacks Project}. Further, as $\Gamma_g$ is a closed immersion and $\mathcal{O}_C$ is flat over $\Spec(C)$ it follows by considering stalks that $(\Gamma_{g})_*\mathcal{O}_C $ is also flat over $\Spec(C)$. Then as $\Gamma_g$ is affine $j_p^* (\Gamma_{g})_* \mathcal{O}_C \cong (\Gamma_{g \circ i_p})_* i_p^* \mathcal{O}_C$ for all $p \in \MaxSpec(C)$ with diagrams \eqref{ClosedBaseChange} by  \cite[Lemma 29.5.1 (Tag 02KE)]{StacksProject}, hence 
\[ \mathbb{L}j_p^*(\Gamma_{g})_* \mathcal{O}_C \cong j_p^* (\Gamma_{g})_* \mathcal{O}_C \cong (\Gamma_{g \circ i_p})_*i_p^* \mathcal{O}_C \cong \mathcal{O}_{g(p)}. \] Secondly, for any $\mathcal{F} \in \Perf(X)$ both $\mathbb{L}g^* \mathcal{F}$ and its derived dual $\mathbb{R}\Hom_C(\mathbb{L}g^*\mathcal{F},\mathcal{O}_C)$ are in $\Perf(C)$ so $\mathbb{R}\Hom_{X^C}(\mathbb{L}\rho^{X*}\mathcal{F}, (\Gamma_{g})_* \mathcal{O}_C) \cong \mathbb{R}\Hom_{C}(\mathbb{L}g^* \mathcal{F},\mathcal{O}_C)  \in \Perf(C)$. Hence $\mu_C$ is well defined as $(\Gamma_{g})_* \mathcal{O}_{\Spec(C)} \in \mathcal{F}_A(C)$ for any $g \in \Hom_{\mathfrak{Sch}}(\Spec(C),X)$. It is natural as if $B,C \in \mathfrak{R}$ with a morphism $u:\Spec(B) \rightarrow \Spec(C)$ and $g:\Spec(C) \rightarrow X \in \Hom_{\mathfrak{Sch}}(\Spec(C),X)$ we have the diagram 
\begin{align*}
\begin{tikzpicture}[scale=1.2] [bend angle=20]
\node (C3) at (0,1.6)  {$\Spec(B)$};
\node (C4) at (2,1.6)  {$\Spec(C)$};
\node (C5) at (0,3.2)  {$X^B$};
\node (C6) at (2,3.2)  {$X^C$};
\node (C7) at (4,3.2) {$X$};
\node (C8) at (4,1.6) {$\mathbb{C}$};
\draw [->] (C3) to node[above]  {\scriptsize{$u$}} (C4);
\draw [->] (C5) to node[above]  {\scriptsize{$v$}} (C6);
\draw [->] (C6) to node[right]  {\scriptsize{$\rho^C$}} (C4);
\draw [->] (C5) to node[right]  {\scriptsize{$\rho^B$}} (C3);
\draw [->] (C4) to node[above]  {$$} (C8);
\draw [->] (C6) to node[above]  {\scriptsize{$\rho^X$}} (C7);
\draw [->] (C7) to node[right]  {\scriptsize{$\rho$}} (C8);
\draw [->] (C4) to node[below]  {\scriptsize{$g$}} (C7);
\draw [->, bend left] (C4) to node[left]  {\scriptsize{$\Gamma_g$}} (C6);
\draw [->, bend left] (C3) to node[left]  {\scriptsize{$\Gamma_{g\circ u}$}} (C5);
\end{tikzpicture}
\end{align*}
where $g=\rho^X \circ \Gamma_g$, $g \circ u= \rho^X  \circ v \circ \Gamma_{g \circ u} $ and the squares can be seen to be pullback squares using the universal property of pullback squares and the fact that $\rho^B \circ \Gamma_{g \circ u}$ is the identity. As above, as $\Gamma_{g}$ and $\Gamma_{g \circ u }$ are closed immersions 
\[ (\Gamma_{g \circ u})_*   u^* \E  \cong v^*  (\Gamma_{g})_*  \E
\]
for any $\E \in \Coh(\Spec(C))$ by \cite[Lemma 29.5.1 (Tag 02KE)]{StacksProject}.
Hence
\begin{equation*}
\mu_B(g\circ u) \cong \Gamma_{(g \circ u) *}\mathcal{O}_{B} \cong   \Gamma_{(g \circ u) *}u^*\mathcal{O}_{C} \cong  v^* (\Gamma_{g})_*\mathcal{O}_{C} \cong v^*\mu_{B}(g).
\end{equation*}

To show it is a natural isomorphism we need to check that $\mu_B$ is bijective for all $B \in \mathfrak{R}$. We do this now by constructing an inverse $\nu_B$.  For $B \in \mathfrak{R}$, given $\mathcal{E} \in \mathcal{F}_{X}(B)$ we consider its schematic support $Z$ and we then have the diagram
\begin{align*}
\begin{tikzpicture} [bend angle=0]
\node (C2) at (2,0)  {$\Spec(B)$};
\node (C3) at (0,2)  {$Z$};
\node (C4) at (2,2)  {$X^B$};
\node (B1) at (4,0) {$\Spec(\mathbb{C})$};
\node (B2) at (4,2) {$X$};
\draw [->] (C4) to node[above]  {\scriptsize{$\rho^{X}$}} (B2);
\draw [->] (C4) to node[right]  {\scriptsize{$\rho^B$}} (C2);
\draw [right hook->] (C3) to node[above]  {\scriptsize{$\iota$}} (C4);
\draw [->] (C3) to node[left]  {\scriptsize{$\psi$}} (C2);
\draw [->] (C2) to node[below]  {} (B1);
\draw [->] (B2) to node[right]  {\scriptsize{$\rho$}} (B1);
\end{tikzpicture}
\end{align*}
where we define $\psi=  \rho^B \circ \iota$. We recall that $\psi$ is an isomorphism from Lemma \ref{HilbertModuliLemma} ii), and we then consider the map $\rho^X \circ \iota \circ \psi^{-1}: \Spec(B) \rightarrow X \in \Hom_{\mathfrak{Sch}}(\Spec(B),X)$, and our inverse is defined by sending $\E \in \mathcal{F}_{X}(B)$ to this element of $\Hom_{\mathfrak{Sch}}(\Spec(B),X)$:
\begin{align*}
\nu_B: \,&  \mathcal{F}_{X}(B)   \rightarrow \Hom_{\mathfrak{Sch}}(\Spec(B),X)\\
 &\E  \mapsto  \left( \rho_X \circ \iota \circ \psi^{-1} :\Spec(B) \rightarrow X \right).
\end{align*}
Finally we note that this is an inverse, as 
\begin{equation*}
\nu_B \circ \mu_B (g)= \nu_B({\Gamma_g}_*\mathcal{O}_B) = g
\end{equation*}
and
\begin{equation*}
\mu_B \circ \nu_B ( \E) = \mu_B \left( (\rho_X \circ \iota \circ \psi^{-1}): \Spec(B) \rightarrow X^B \right)= {\Gamma_{(\rho_X \circ \iota \circ \psi^{-1})}}_*\left(\mathcal{O}_{B} \right)
\end{equation*}
where we note that ${\Gamma_{(\rho_X \circ \iota \circ \psi^{-1})}}_* (\mathcal{O}_{B})$ is equivalent to $\E$ in $\mathcal{F}_X(B)$ by Lemma \ref{HilbertModuliLemma} iii). Hence $\Hom_{\mathfrak{Sch}}(-,X)$ is naturally isomorphic $\mathcal{F}_{X}$.

Finally, under this identification the identity morphism $id \in \Hom_{\mathfrak{Sch}}(X,X)$ is mapped to the bundle $\Gamma_{id*} \mathcal{O}_X =\Delta_* \mathcal{O}_X$, so this is the tautological element.

\end{proof}

\begin{TautologicalBundleRemark} 
Combining Theorems \ref{BigTheorem3} and \ref{Representability} we can deduce that if there exists a dimension vector $d$ and stability condition $\theta$ such that $\mathbb{R}\Hom_X(T,-)$ and $T \otimes_{A}^{\mathbb{L}}(-)$ restrict to bijections between $\mathcal{F}_X(\mathbb{C})$ and $\mathcal{F}_A(\mathbb{C})$ then $X$ is a fine moduli space for the functor $\mathcal{F}_A$. In particular $d$ must be indivisible and $\theta$ must be generic. Further, in this situation the tautological bundle on $X$ is in fact $T^{\vee}$, as can be seen by translating the tautological element $\Delta_*\mathcal{O}_X$ across the natural isomorphism between $\mathcal{F}_X$ and $\mathcal{F}_A$, so $\End_X(T^{\vee}) \cong \End_X(T)^{\op}=A$ and the dual of the tautological bundle is the tilting bundle $T$.

\end{TautologicalBundleRemark}
\section{Applications} \label{Applications}

Let $\pi:X \rightarrow \Spec(R)$ be a projective morphism of finite type schemes over $\mathbb{C}$, suppose $X$ has a tilting bundle $T$, and suppose that $A=\End_{X}(T)^{\op}$ is presented as a quiver with relations.  In this section we will introduce an indivisible dimension vector $d_T$ and generic stability condition $\theta_T$ defined by a decomposition of the tilting bundle and give general conditions for the map $\eta:\mathcal{F}_X \rightarrow \mathcal{F}_A$ introduced in the previous sections to be a natural isomorphism for this stability condition and dimension vector.  We will then use these general conditions to produce the applications outlined in the introduction.

\subsection{Dimension Vectors and Stability} \label{Dimension Vectors and Stability} We introduce a certain dimension vector and stability condition defined from a decomposition of a tilting bundle and then, using Theorem \ref{BigTheorem3}, we give criterion for $\eta$ to be a natural isomorphism with respect to this stability condition and dimension vector. In order to do this we make the following assumption on $T$, a tilting bundle on a scheme $X$.
\begin{TiltingAssumptions} \label{Assumptions}  
 The tilting bundle $T$ has a decomposition into non-isomorphic indecomposables $T=\bigoplus_{i=0}^n E_i$ such that there is a unique indecomposable, $E_0$, isomorphic to $\mathcal{O}_X$. 
\end{TiltingAssumptions}
We then consider a presentation of $A=\End_X(T)^{\op}$ as the path algebra of a quiver with relations such that each indecomposable $E_i$ corresponds to a vertex $i$ of the quiver, as in Section \ref{Quivers and Tilting bundles}. In particular the 0 vertex in the quiver corresponds to the summand $\mathcal{O}_X$.

\begin{DimensionDefinition}
Suppose $T$ is a tilting bundle $T$ with decomposition $T= \bigoplus_{i=0}^n E_i$. 
\begin{itemize}
\item[i)] The dimension vector $d_T$ is defined by
\[ d_T(i) = \rk E_i. \]
In particular $d_T(0)=1$ as $E_0$ is assumed to be isomorphic to $\mathcal{O}_X$ so $d_T$ is indivisible.
\item[ii)] The stability condition $\theta_{T}$ is defined by 
\[ \theta_T(i) = \left\{ \begin{array}{cc}
         -\sum_{i \neq 0} \rk E_i & \mbox{if $i=0$};\\
         1 & \mbox{otherwise}.\end{array} \right. \] 
\end{itemize}
\end{DimensionDefinition}

\begin{StabilityConditionEquiv} \label{StabilityConditionEquiv}The stability condition $\theta_T$ has the following properties:
\begin{itemize}
\item[i)]
Let $P_0:=\mathbb{R}\Hom_{X}(T,\mathcal{O}_X)$ and $M$ be an $A$-module with dimension vector $d_T$. Then $\Hom_{A}(P_0,M)$ is one dimensional, and $M$ is $\theta_T$-stable if and only if there is a surjection $P_0 \rightarrow M \rightarrow 0$.

\item[ii)]
The stability $\theta_T$ is generic for $A$-modules of dimension $d_T$.
\end{itemize}
\end{StabilityConditionEquiv}
\begin{proof}
The $A$-module $P_0$ is the projective module consisting of paths in the quiver starting at the vertex 0. For any representation $M$ with dimension vector $d_T$ a homomorphism from $P_0$ to $M$ is determined by the image of the trivial path $e_0 \in P_0$ in the vector space $\mathbb{C} \subset M$ at vertex $0$, which we denote by $1_0$. This is as any path $p$ starting at 0 must be sent to the evaluation in $M$ of the linear map corresponding to $p$ on the element $1_0$. Hence $\Hom_{A}(P_0,M)=\mathbb{C}$, and any nonzero element of $\Hom_{A}(P_0,M)$ is surjective precisely when the linear maps in $M$ corresponding to paths starting at 0 form a surjection from the vector space at the zero vertex onto $M$. By the definition of $\theta_T$ the module $M$ is $\theta_T$-semistable if and only if any proper submodule $N$ has $d_N(0)=0$, and this property is equivalent to the linear maps in $M$ corresponding to paths starting at 0 forming a surjection. This proves part i).

We now prove ii). It is clear by the definitions of $\theta_T$ and $d_T$ that any dimension $d_T$ module $M$ can have no proper submodules $N \subset M$ such that $\theta_T(N)=0$ as if $N$ is a nontrivial submodule, either $d_N(0)=0$ and $\theta_T(N)>0$, or $d_N(0)=1$ and $N=M$. 
\end{proof}

We now give conditions for $\eta:\mathcal{F}_X \rightarrow \mathcal{F}_A$ to be a natural isomorphism for this stability condition and dimension vector. We note that there is an abelian category $\mathcal{A}$ corresponding to the abelian category $A$-$\mod$ under the tilting equivalence between $D^b(X)$ and $D^b(A)$ such that $T$ is a projective generator of $\mathcal{A}$. Then $\mathbb{R}\Hom_{X}(T,-)$  and $ T \otimes_{A}^{\mathbb{L}} (-) $ define an equivalence of abelian categories between $\mathcal{A}$ and $A$-$\mod$. Our conditions are defined on this category $\mathcal{A}$.

\begin{Surjectivity}\label{Surjectivity} Take the dimension vector $d_T$ and  stability condition $\theta_T$ as above. Suppose the following conditions hold:
\begin{itemize}
\item[i)] The structure sheaf $\mathcal{O}_X$ is in $\mathcal{A}$, and for all closed points $x \in X$ the skyscraper sheaf $ \mathcal{O}_x$ is in  $\mathcal{A}$.
\item[ii)] For all closed points $x \in X$ there are surjections $\mathcal{O}_{X} \rightarrow \mathcal{O}_x \rightarrow 0$ in $\mathcal{A}$.
\end{itemize}
Then $\eta$ is a well defined natural transformation and $\eta_B$ is injective for all $B \in \mathfrak{R}$. Suppose further that the following condition also holds:
\begin{itemize}
\item[iii)] The set 
\begin{align*}
     S:=   \left. \left\{ \E \in \mathcal{A} \, \, \begin{array}{ | l } \bullet  \text{ $\mathbb{R}\Hom_{X}(T,\E)$ has dimension vector $d_T$.} \\ \bullet\text{ $\Hom_{\mathcal{A}}(\E,\mathcal{O}_x)= 0 $ for all closed points $x \in X$.} 
\end{array}  \right\}  \right.
\end{align*} 
is empty.
\end{itemize}
Then $\eta$ is a natural isomorphism.
\end{Surjectivity}

\begin{proof}
We first assume that conditions i) and ii) hold and prove that $\eta_{\mathbb{C}}$ is well defined. Then it follows from Lemmas \ref{VitalLemma2} and \ref{VitalLemma3} that $\eta$ is a natural transformation and $\eta_B$ is injective for all $B \in \mathfrak{R}$.

Any element of $\mathcal{F}_{X}(\mathbb{C})$ is a skyscraper sheaf on $X$ up to isomorphism. For any closed point $x \in X$ the $A$-module $\mathbb{R}\Hom_{X}(T,\mathcal{O}_x)$ has dimension vector $d_T$, hence the map $\eta_{\mathbb{C}}$ is well defined if and only if all $\mathbb{R}\Hom_{X}(T,\mathcal{O}_x)$ are $\theta_T$-semistable $A$-modules. By condition i) they are $A$-modules. By considering the surjections of condition ii), $\mathcal{O}_{X} \rightarrow \mathcal{O}_x \rightarrow 0$ in $\mathcal{A}$, and applying the abelian equivalence $\mathbb{R}\Hom_{X}(T,-)$ we see that all $\mathbb{R}\Hom_X(T,\mathcal{O}_x)$ are $\theta_T$-stable by Lemma \ref{StabilityConditionEquiv} i). Hence $\eta_{\mathbb{C}}$ is well defined.

We now also assume that condition iii) holds and prove that $\eta_{\mathbb{C}}$ is also surjective with inverse $T \otimes_A^{\mathbb{L}}(-)$. It then follows from Theorem \ref{BigTheorem3} that $\eta$ is a natural isomorphism. Take an $A$-module, $M$, with dimension vector $d_T$ and which is $\theta_T$-stable. As $M$ is $\theta_T$-stable by Lemma \ref{StabilityConditionEquiv} ii) there is a surjection 
\begin{equation*}
P_0 \rightarrow M \rightarrow 0
\end{equation*}
which under the abelian equivalence gives an exact sequence in $\mathcal{A}$
\begin{equation*}
\mathcal{O}_{X} \rightarrow \E \rightarrow 0
\end{equation*}
where $\E \cong M \otimes_{A}^{\mathbb{L}} T \in D^b(X)$. Then by condition iii) there must be some closed point $x \in X$ such that $\Hom_{\mathcal{A}}(\E,\mathcal{O}_x) \neq 0$. We then apply $\Hom_{\mathcal{A}}(-,\mathcal{O}_x)$ to the surjection $\mathcal{O}_{X} \rightarrow \E \rightarrow 0$ to obtain an injection
\begin{equation*}
0 \rightarrow \Hom_{\mathcal{A}}(\E,\mathcal{O}_x) \rightarrow \Hom_{\mathcal{A}}(\mathcal{O}_{X},\mathcal{O}_x)=\mathbb{C}
\end{equation*}
and hence the surjection $\mathcal{O}_{X} \rightarrow \mathcal{O}_x \rightarrow 0$ factors through $\E$, and there is a surjection $\E \rightarrow \mathcal{O}_x \rightarrow 0$.  We then apply the abelian equivalence functor $\mathbb{R}\Hom_{X}(T,-)$ to obtain a surjection of finite dimensional $A$-modules
\begin{equation*}
M \rightarrow \mathbb{R} \Hom_{X}(T,\mathcal{O}_x)\rightarrow 0
\end{equation*}
and  by comparing dimension vectors we see that the map is an isomorphism, hence that $ \mathbb{R}\Hom_{X}(T,\mathcal{O}_x) \cong M$ and $T \otimes^{\mathbb{L}}_A M\cong \mathcal{O}_x$.
\end{proof}

\begin{GeneralClosedImmersion} \label{GeneralClosedImmersion} 
 Let $\pi:X \rightarrow \Spec(R)$ be a projective morphism of finite type schemes over $\mathbb{C}$. Let $T$ be a tilting bundle on $X$ which defines an equivalence of an abelian category $\mathcal{A}$ with $A$-$\mod$, where $A=\End_{X}(T)^{\op}$. Choose the stability condition $\theta_{T}$ and dimension vector $d_{T}$ as above, define $\mathcal{F}_{A}=\mathcal{F}_{A,d_T,\theta_T}^{ss}$, and assume that conditions i) and ii) of Lemma \ref{Surjectivity} hold for $\mathcal{A}$. Then:
\begin{itemize}
\item[i)] The map $\eta: \mathcal{F}_X \rightarrow \mathcal{F}_{A}$ defined in Section \ref{Natural Transformation} is a natural transformation and induces a morphism $f:X \rightarrow \mathcal{M}^{ss}_{d_{T},\theta_{T}} $ between $X$ and the quiver GIT quotient of $A$ for stability condition $\theta_{T}$ and dimension vector $d_{T}$. This morphism is a monomorphism in the sense of \cite[Definition 25.23.1 (Tag 01L2)]{Stacks Project}.
\item[ii)] If condition iii) of Lemma \ref{Surjectivity} also holds for $\mathcal{A}$ then the morphism $f$ is an isomorphism.
\end{itemize}
\end{GeneralClosedImmersion}
\begin{proof}
We note that $\mathcal{M}^{ss}_{d_{T},\theta_{T}} = \mathcal{M}^{s}_{d_{T},\theta_{T}}$ as $\theta_T$ is generic by Lemma \ref{StabilityConditionEquiv} ii) and that $\mathcal{M}^{s}_{d_{T},\theta_{T}} $ is a fine moduli space for $\mathcal{F}_A$ by Theorem \ref{sModuliGIT} as the dimension vector $d_T$ is indivisible. The map $\eta: \mathcal{F}_{X} \rightarrow \mathcal{F}_A$ is a natural transformation as conditions i) and ii) of Lemma \ref{Surjectivity} hold for $\mathcal{A}$. It then follows that there is a corresponding morphism $f:X \rightarrow \mathcal{M}^{ss}_{d_{T},\theta_{T}}$  as  $\mathcal{F}_{A}$ is represented by $\mathcal{M}^{ss}_{d_{T},\theta_{T}}$ and $\mathcal{F}_X$ is represented by $X$ by Theorem \ref{Representability}. For all $B \in \mathfrak{R}$ the map $\eta_B$ is injective by Lemma \ref{Surjectivity}, hence the corresponding morphism, $f$, is a monomorphism.

If condition iii) of Lemma \ref{Surjectivity} also holds for  $\mathcal{A}$ then $\eta$ is actually a natural isomorphism by Lemma \ref{Surjectivity}. Hence $f$ is an isomorphism, proving ii).
\end{proof}

\begin{MonomorphismRemark}
While we make no further use of the monomorphism property we note that it can be a useful notion as proper monomorphisms are exactly closed immersions, \cite[Lemma 40.7.2 (Tag 04XV)]{Stacks Project}, and \'etale monomorphisms are exactly open immersions, \cite[Theorem 40.14.1 (Tag 025G)]{Stacks Project}.
\end{MonomorphismRemark}

\subsection{One Dimensional Fibres} \label{One Dimensional Fibres}
To apply Lemma \ref{Surjectivity} and Corollary \ref{GeneralClosedImmersion} we need a class of varieties with tilting bundles such that we understand the abelian categories $\mathcal{A}$. Such a class was introduced in Theorem \ref{TiltingBundles1dim}; if $\pi:X \rightarrow \Spec(R)$ is a projective morphism of Noetherian schemes such that $\mathbb{R}\pi_* \mathcal{O}_X \cong \mathcal{O}_R$ and the fibres of $\pi$ have dimension $\le 1$ then there exist tilting bundles $T_i$ on $X$ such that the abelian category $\mathcal{A}$ is $^{-i}\Per(X/R)$, defined as follows.

\begin{Perverse}[{\cite[Section 3]{3Dflops}}] \label{Perverse} Let $\pi:X \rightarrow  \Spec(R)$ be a projective morphism of Noetherian schemes such that $\mathbb{R}\pi_*\mathcal{O}_X  \cong \mathcal{O}_R$ and $\pi$ has fibres of dimension $\le 1$. Define $\mathfrak{C}$ to be the abelian subcategory of $\Coh \, X$ consisting of $\mathcal{F} \in \Coh\, X$ such that $\mathbb{R}\pi_*\mathcal{F} \cong 0$. For $i=0,1$ the abelian category $^{-i}\Per(X/R)$ is defined to contain $\E \in D^b(X)$ which satisfy the following conditions:
\begin{itemize}
\item[i)] The only non-vanishing cohomology of $\E$ lies in degrees $-1$ and $0$.
\item[ii)] $\pi_*\mathcal{H}^{-1}(\E)=0$ and  $\mathbb{R}^1 \pi_* \mathcal{H}^0(\E)=0$, where $\mathcal{H}^j$ denotes taking the $j^{th}$ cohomology sheaf.
\item[iii)] For $i=0$, $\Hom_X(C,\mathcal{H}^{-1}(\E))=0$ for all $C \in \mathfrak{C}$.
\item[iv)] For $i=1$, $\Hom_X(\mathcal{H}^0(\E),C)=0$ for all $C \in \mathfrak{C}$.
\end{itemize}
We note that the abelian categories $^{-i}\Per(X/R)$ are hearts of $t$-structures on $D^b(X)$ so short exact sequences in $^{-i}\Per(X/R)$ correspond to triangles in $D^b(X)$ whose vertices are in $^{-i}\Per(X/R)$.
\end{Perverse}
We note the following property of morphisms in $^{-i}\Per(X/R)$.
\begin{PerverseHoms} \label{PerverseHoms} Let $\pi:X \rightarrow \Spec(R)$ be a projective morphism of finite type schemes over $\mathbb{C}$, and suppose that $X$ has a tilting bundle $T$ that induces an abelian equivalence between $^0\Per(X/R)$ and $A$-$\mod$ where $A=\End_X(T)^{\op}$. Then $\Hom_{{^0\Per(X/R)}}(\E_1,\E_2)\cong \Hom_{D^b(X)}(\E_1,\E_2)$ for $\E_1, \E_2 \in {^0\Per(X/R)}$.
\end{PerverseHoms}
\begin{proof}
Let $\E_1, \E_2 \in {^0\Per(X/R)}$. Then $M_i =\mathbb{R}\Hom_{X}(T,\E_i)$ is an $A$-module for $i=1,2$ and  $\Hom_{\mathcal{A}}(\E_1,\E_2)\cong \Hom_{A}(M_1, M_2) \cong \Hom_{D^b(A)}(M_1, M_2) \cong \Hom_{D^b(X)}(\E_1,\E_2)$ by the abelian and then derived equivalence. 
\end{proof}

Any projective generator of the abelian category $^{-i}\Per(X/R)$ gives a tilting bundle $T_i$ with the properties defined in Theorem \ref{TiltingBundles1dim}, and we can assume that such a tilting bundle contains $\mathcal{O}_X$ as a summand by the following proposition.

\begin{Progenerators}[{\cite[Proposition 3.2.7]{3Dflops}}] \label{Progenerators} Define $\mathfrak{V}_X$ to be the category of vector bundles $\mathcal{M}$ on $X$ which are generated by global sections and such that $\CH^1(X,\mathcal{M}^{\vee})=0$, and define $\mathfrak{V}_X^{\vee}:=\{ \mathcal{M}^{\vee} : \mathcal{M} \in \mathfrak{V}_X \}$. The projective generators of $^{-1}\Per(X/R)$ are the $\mathcal{M} \in \mathfrak{V}_X$ such that $\wedge^{\rk \mathcal{M}} \mathcal{M}$ is ample and $\mathcal{O}_X$ is a summand of $\mathcal{M}^{\oplus a}$ for some $a \in \mathbb{N}$. The projective generators of $^0 \Per(X/R)$ are the elements of $\mathfrak{V}_X^{\vee}$ which are dual to projective generators of $^{-1}\Per(X/R)$.
\end{Progenerators}

Hence we let $T_i$ be a projective generator of $^i\Per(X/R)$ with a decomposition as required in Assumption \ref{Assumptions}. Then the algebra $A_i=\End_{X}(T_i)^{\op}$ can be presented as a quiver with relations with vertex 0 corresponding to $\mathcal{O}_{X}$ and the stability condition $\theta_{T_i}$ and dimension vector $d_{T_i}$ are well defined.

We now check that the conditions of Lemma \ref{Surjectivity} hold for $^0\Per(X/R)$.

\begin{BijectiveConditions} \label{BijectiveConditions} Let $\pi:X \rightarrow \Spec(R)$ be a projective morphism of finite type schemes over $\mathbb{C}$ such that $\pi$ has fibres of dimension $\le 1$ and $\mathbb{R}\pi_* \mathcal{O}_{X} \cong \mathcal{O}_{R}$. Then the abelian category ${^0\Per(X/R)}$ satisfies conditions i), ii) and iii) of Lemma \ref{Surjectivity}.
\end{BijectiveConditions}

\begin{proof}
We begin by checking $\mathcal{A}$ satisfies conditions i) and ii) of Lemma \ref{Surjectivity}. All skyscraper sheaves $\mathcal{O}_x$ and the structure sheaf $\mathcal{O}_{X}$ are in $\mathcal{A}$ as they satisfy the conditions of Definition \ref{Perverse}. Then, for any $x \in X$, the short exact sequence of sheaves $0 \rightarrow I \rightarrow \mathcal{O}_X \rightarrow \mathcal{O}_x \rightarrow 0$ corresponds to a triangle in $D^b(X)$, and the ideal sheaf $I$ is also in $\mathcal{A}$ as $\mathbb{R}^1\pi_* I = 0$ due to the exact sequence $0 \rightarrow \pi_*I \rightarrow \pi_* \mathcal{O}_X \rightarrow \pi_* \mathcal{O}_x \rightarrow \mathbb{R}^1\pi_*I \rightarrow 0$ where $\pi_*\mathcal{O}_X \cong \mathcal{O}_R$ and the third arrow is a surjection. Hence the map  $\mathcal{O}_X \rightarrow \mathcal{O}_x \rightarrow 0$ is in fact a surjection in $\mathcal{A}$. We then note, for all $x \in X$, that $\Hom_{\mathcal{A}}(\mathcal{O}_{X},\mathcal{O}_x) \cong \Hom_{D^b(X)}(\mathcal{O}_X,\mathcal{O}_x) \cong \Hom_X(\mathcal{O}_X,\mathcal{O}_x) \cong \mathbb{C}$, hence $\Hom_{\mathcal{A}}(\mathcal{O}_{X},\mathcal{O}_x)\cong \mathbb{C}$ corresponding to the map of sheaves $\mathcal{O}_{X} \rightarrow \mathcal{O}_x \rightarrow 0$ which is surjective in $\mathcal{A}$.

To check condition iii)  suppose $S$ is not empty and so there exists some $\E \in S$. In particular,  $M \cong \Hom_{D^b(X)}(T_0,\E)$ has dimension vector $d_{T_0}$ so $\mathbb{R}\pi_*\E \cong \mathcal{O}_y$ for some $y \in \Spec(R)$. As $\E \in \mathcal{A}$ there is a short exact sequence in $\mathcal{A}$
\[
0 \rightarrow \mathcal{H}^{-1}(E)[1] \rightarrow \E \rightarrow \mathcal{H}^0(\E) \rightarrow 0
\]
where $[1]$ is the shift in $D^b(X)$. Hence, for all closed points $x \in X$, there is an injection 
\[
0 \rightarrow \Hom_{\mathcal{A}}(\mathcal{H}^0(\E),\mathcal{O}_x) \rightarrow \Hom_{\mathcal{A}}(\E,\mathcal{O}_x).
\]
Then it follows that $\Hom_{\mathcal{A}}(\mathcal{H}^0(\E),\mathcal{O}_x)=\Hom_{D^b(X)}(\mathcal{H}^0(\E),\mathcal{O}_x)=0$ for all $x \in X$ as by assumption $\Hom_{\mathcal{A}}(\E,\mathcal{O}_x)=0$, and hence $\mathcal{H}^0(\E)=0$ as a nonzero coherent sheaf must be supported somewhere. So $\E=\mathcal{H}^{-1}(\E)[1]$, and we now seek to reach a contradiction to the existence of such an $\E$. The argument below should be thought of as an explicit translation to our setting of the proof of Nakamura's conjecture for the $G$-Hilbert scheme in \cite[Section 8]{McKayBKR} which derives a contradiction between the facts that the Euler pairing of a coherent sheaf shifted by [1] with a very ample line bundle must be negative whereas the Euler pairing of a $G$-cluster with any locally free sheaf must be positive. 

We begin by noting in particular that $\pi_*\mathcal{H}^{-1}(\E)=0$ and $\mathbb{R}^1 \pi_*\mathcal{H}^{-1}(\E)=\mathcal{O}_y$. By \cite[Lemma 3.1.3]{3Dflops} there is an injection of sheaves
\[
0 \rightarrow \mathcal{H}^{-1}(\E) \rightarrow \mathcal{H}^{-1}(\pi^{!}\mathcal{O}_y)
\]
and hence $\mathcal{H}^{-1}(\E)$ is set-theoretically supported on $\pi^{-1}(y)$. In particular $y$ corresponds to a maximal ideal $\frak{m}_y$ of $R$ and we consider the completion $R \rightarrow \hat{R}=\varprojlim ( R/\frak{m}_y^n )$. This produces the following pullback diagram 
\begin{align*}
\begin{tikzpicture} [bend angle=0]
\node (C3) at (0,1.6)  {$\Spec(\hat{R})$};
\node (C4) at (2,1.6)  {$\Spec(R)$};
\node (C5) at (0,3.2)  {$Y$};
\node (C6) at (2,3.2)  {$X$};
\draw [->] (C3) to node[above]  {\scriptsize{$i$}} (C4);
\draw [->] (C5) to node[above]  {\scriptsize{$j$}} (C6);
\draw [->] (C6) to node[right]  {\scriptsize{$\pi$}} (C4);
\draw [->] (C5) to node[left]  {\scriptsize{$\hat{\pi}$}} (C3);
\end{tikzpicture}
\end{align*}
where $Y$ is the formal fibre  $Y:=\varprojlim (\Spec(R/\frak{m}_y^n) \times_{\Spec(R)} X)$, the morphisms $i$ and $j$ are both flat and affine, and the morphism $\hat{\pi}$ is projective. Then we have the following isomorphism, where we recall that the morphisms $i$ and $j$ are both flat and affine so we need not derive them,
\begin{align*}
\mathbb{R}\Hom_X(T_0,j_*j^*\E)  &\cong i_*\mathbb{R}\Hom_{Y}(j^*T_0,j^*\E) \tag{$j_*,j^*$ adjoint pair} \\
&   \cong i_*\mathbb{R}\hat{\pi}_{*} j^* \RCalHom_{X}(T_0,\E) \tag{Lemma \ref{DualPullback}} \\ 
& \cong i_* i^* \mathbb{R}\Hom_{X}(T_0,\E). \tag{Flat base change}
\end{align*}
Then as $M\cong \mathbb{R}\Hom_{X}(T_0,\E)$ is finite dimensional and supported on $\mathfrak{m}_y$ it follows that completion in $\mathfrak{m}_y$ followed by restriction of scalars acts as the identity, see \cite[Theorem 2.13]{Eisenbud} and \cite[Lemma 2.5]{CMReps}, hence $ i_*i^*M := \hat{R} \otimes_R M \cong M$. We deduce that  $\mathbb{R}\Hom_X(T_0,j_*j^*\E) \cong \mathbb{R}\Hom_{X}(T_0,\E)$, and so $\E \cong j_*j^*\E$ as $T_0$ is a tilting bundle. Finally we can define $\mathcal{G}:= j^* \mathcal{H}^{-1}(\E)$ with the property that $j_* \mathcal{G}[1] \cong \E $.

We now note that by Lemma \ref{Progenerators} there exists $P \in \mathfrak{V}_X$ such that $T_0=P^{\vee}$.  We then note that as $P$ is a vector bundle generated by global sections so is $j^* P$, hence as $\hat{R}$ is a complete local ring there exists a short exact sequence
\[
0 \rightarrow \mathcal{O}_{Y}^{\oplus d-1} \rightarrow j^* P \rightarrow \wedge^d  j^*P \rightarrow 0
\]
by \cite[Lemma 3.5.1]{3Dflops}, where $d=\rk P=\rk  j^* P$. Also, as $P \in \mathfrak{V}_X$, the line bundle $\wedge^d P$ is ample and so the line bundle $\mathcal{L}:=\wedge^d j^* P \cong j^* \wedge^d P$ is also ample as $j$ is affine. Then by Serre vanishing, \cite[\textrm{III} Theorem 5.2]{HartshorneAG}, there exists some $N>0$ such that $\Hom_{D^b(Y)}(\mathcal{L}^{\otimes - N},\mathcal{G}[1]) \cong \Ext^1_Y(\mathcal{O}_Y,\mathcal{L}^{\otimes N} \otimes \mathcal{G})=0$. As $j^*P$ is generated by global sections the vector bundle $j^*P^{\oplus N}$ is also generated by global sections so again there exists a short exact sequence
\[
0 \rightarrow \mathcal{O}_Y^{\oplus Nd-1} \rightarrow \left(j^*P \right)^{\oplus N} \rightarrow \mathcal{L}^{\otimes N} \rightarrow 0.
\]
by \cite[Lemma 3.5.1]{3Dflops}. Dualising this we obtain the short exact sequence
\[
0 \rightarrow \mathcal{L}^{\otimes -N} \rightarrow \left(j^*T_0 \right)^{\oplus N} \rightarrow \mathcal{O}_Y^{\oplus Nd-1} \rightarrow 0,
\]
where $(j^* P)^{\vee} \cong j^* (P^{\vee})$ by Lemma \ref{DualPullback}. As $\Hom_{D^b(Y)}(\mathcal{L}^{\otimes - N},\mathcal{G}[1])=0$ applying $\Hom_{D^b(Y)}(-,\mathcal{G}[1])$ to this sequence produces an exact sequence
\begin{equation} \tag{$\dagger$} \label{Sequence}
\Hom_{D^b(Y)}(\mathcal{O}_Y,\mathcal{G}[1])^{\oplus Nd-1} \rightarrow \Hom_{D^b(Y)}(j^*T_0,\mathcal{G}[1])^{\oplus N} \rightarrow 0.
\end{equation}
Then 
\begin{align*}
\dim_{\mathbb{C}} \Hom_{D^b(Y)}(\mathcal{O}_Y,\mathcal{G}[1]) 
&= \dim_{\mathbb{C}}\Hom_{D^b(Y)}(\mathbb{L}j^*\mathcal{O}_X,\mathcal{G}[1]) \tag{$\mathcal{O}_Y \cong \mathbb{L}j^*\mathcal{O}_X$} \\
&= \dim_{\mathbb{C}}\Hom_{D^b(X)}(\mathcal{O}_X, \mathbb{R}j_{*}\mathcal{G}[1]) \tag{$\mathbb{L}j^*,\mathbb{R}j_{*}$  adjoint pair} \\
&= \dim_{\mathbb{C}}\Hom_{D^b(X)}(\mathcal{O}_X, j_{*}\mathcal{G}[1]) \tag{As $j$ affine $\mathbb{R}j_*=j_*$} \\
&= \dim_{\mathbb{C}}\Hom_{D^b(X)}(\mathcal{O}_X,\E) 
=1 
\intertext{and}
 \dim_{\mathbb{C}} \Hom_{D^b(Y)}(j^*T_0,\mathcal{G}[1])   &= \dim_{\mathbb{C}} \Hom_{D^b(Y)}(\mathbb{L}j^*T_0,\mathcal{G}[1]) \tag{$T_0$ locally free} \\
 &= \dim_{\mathbb{C}} \Hom_{D^b(X)}(T_0,\mathbb{R}j_{*}\mathcal{G}[1]) \tag{$\mathbb{L}j^*,\mathbb{R}j_{*}$  adjoint pair} \\
 &= \dim_{\mathbb{C}} \Hom_{D^b(X)}(T_0,j_{*}\mathcal{G}[1]) \tag{As $j$ affine $\mathbb{R}j_*=j_*$} \\
& = \dim_{\mathbb{C}} \Hom_{D^b(X)}(T_0,\E)\\
&=  \dim_{\mathbb{C}} M  =d
\end{align*}
as $M \cong \Hom_{D^b(X)}(T_0,\E)$ has dimension vector $d_{T_0}$ and $d=\rk T_0$. Comparing the dimensions in the sequence \eqref{Sequence} we find a contradiction since a $Nd-1$ dimensional space cannot surject onto an $Nd$ dimensional space. Hence such an $\E$ cannot exist and so $S$ is empty.
\end{proof}

Combining this theorem with Corollary \ref{GeneralClosedImmersion} gives us the following result, showing that in this situation schemes can be reconstructed as fine moduli spaces by quiver GIT.

\begin{BigTheorem1} \label{BigTheorem1} 
 Let $\pi:X \rightarrow \Spec(R)$ be a projective morphism of finite type schemes over $\mathbb{C}$ such that $\pi$ has fibres of dimension $\le 1$ and $\mathbb{R}\pi_* \mathcal{O}_{X} \cong \mathcal{O}_{R}$. Let $T_0$ be a tilting bundle which is a projective generator of $^0\Per(X/R)$ as defined by Theorem \ref{TiltingBundles1dim}, define $A_0=\End_{X}(T_0)^{\op}$, and choose the stability condition $\theta_{T_0}$ and dimension vector $d_{T_0}$ as above. Then $X$ is the fine moduli space of the quiver representation moduli functor for $A_0=\End_{X}(T_0)^{\op}$ with dimension vector $d_{T_0}$ and stability condition $\theta_{T_0}$ and the tautological bundle is the tilting bundle $T_0^{\vee}$.
\end{BigTheorem1}

\subsection{Example: Flops} \label{Flops}
The class of varieties considered in Section \ref{One Dimensional Fibres} were originally motivated by flops in the minimal model program.  In the paper \cite{FlopsDerivedCategories} Bridgeland proves that smooth varieties in dimension three which are related by a flop are derived equivalent, and in the process constructs the flop of such a variety as a moduli space of perverse point sheaves. In this section we show that this moduli space construction can in fact be done using quiver GIT. Recall the following theorem.

\begin{VanDenBerghFlops}[{\cite[Theorems 4.4.1, 4.4.2]{3Dflops}}] \label{VanDenBerghFlops} Suppose $\pi:X\rightarrow \Spec(R)$ is a projective birational map of quasiprojective Gorenstein varieties of dimension $\ge 3$, with $\pi$ having fibres of dimension $\le 1$, the exceptional locus of $\pi$ having codimension $\ge 2$, and $Y$ having canonical hypersurface singularities of multiplicity $\le 2$. Then the flop $\pi': X'\rightarrow \Spec(R)$ exists and is unique. Further $X$ and $X'$ are derived equivalent such that $^{-1}\Per(X/R)$ corresponds to $^{0}\Per(X'/R)$.  In particular, for a tilting bundle $T_1$ on $X$ which is a projective generator of $^{-1}\Per(X/R)$ there is a tilting bundle $T_0'$ on $X'$ which is a projective generator of  $^{0}\Per(X'/R)$ such that $A_1=\End_{X}(T_1)^{\op} \cong \End_{X'}(T_0')^{\op}=A_0'$ and $\pi_* T_1 \cong \pi'_* T_0'$.
\end{VanDenBerghFlops}

 We refer the reader to \cite[Theorem 4.4.1]{3Dflops} for the definition of a flop in this setting. The results from the previous sections now imply the following corollary, showing that the variety $X$ and its flop $X'$ can both be constructed as quiver GIT quotients from tilting bundles on $X$.

\begin{FlopGIT} \label{FlopGIT}
Suppose we are in the situation of Theorem \ref{VanDenBerghFlops}. Then $X$ is the quiver GIT quotient of $A_0=\End_{X}(T_0)^{\op}$ for stability condition $\theta_{T_0}$ and dimension vector $d_{T_0}$, and $X'$ is the quiver GIT quotient of $A_1=\End_{X}(T_1)^{\op}$ for stability condition $\theta_{T_1}$ and dimension vector $d_{T_1}$.
\end{FlopGIT}
\begin{proof}
Corollary \ref{BigTheorem1} tells us both that $X$ is the quiver GIT quotient of $A_0$ for stability condition $\theta_{T_0}$ and dimension vector $d_{T_0}$, and that $X'$ is the quiver GIT quotient of $A_0' =\End_{X'}(T_0')^{\op}$ for stability condition $\theta_{T_0'}$ and dimension vector $d_{T_0'}$. We now relate $A_0'$, $\theta_{T_0'}$ and $d_{T_0'}$ to $A_1$, $\theta_{T_1}$ and $d_{T_1}$.

We note that by Theorem \ref{VanDenBerghFlops} $A_0'\cong A_1$, and we choose a presentation of $A_1$ as a quiver with relations matching that of $A_0'$ in order to  identify the  stability condition and dimension vector matching $\theta_{T_0'}$ and $d_{T_0'}$. In particular there is a decomposition of $T_1=\bigoplus_{i=0}^n E_i$ and $T_0'=\bigoplus_{i=0}^n E_i'$ such that  $\pi_* E_i \cong \pi'_* E_i'$. We note that under this correspondence the vertices corresponding to $\mathcal{O}_{X}$ and $\mathcal{O}_{X'}$ correspond by \cite[Lemma 4.2.1]{3Dflops} as $\pi_* \mathcal{O}_{X}\cong \pi'_* \mathcal{O}_{X'} \cong \mathcal{O}_R$, and since $\pi$ and $\pi'$ are birational $\rk_XE_i=\rk_R\pi_*E_i=\rk_R \pi'_* E_i' = \rk_{X'}E_i'$.  Hence $A_0' \cong A_1$, $d_{T_0'}=d_{T_1}$ and $\theta_{T_0'}=\theta_{T_1}$ so $X'$ is the quiver GIT quotient of $A_1=\End_{X}(T_1)^{\op}$ for stability condition $\theta_{T_1}$ and dimension vector $d_{T_1}$.
\end{proof}

\subsection{Example: Resolutions of Rational Singularities} \label{ResolutionsRational} We give a further application of Theorem \ref{BigTheorem1} to the case of rational singularities, extending and recapturing several well-known examples. 

\begin{resolution}
 Let $Y$ be a (possibly singular) variety. A smooth variety $X$ with a projective birational map $\pi: X \rightarrow Y$ that is bijective over the smooth locus of $Y$ is called a \emph{resolution} of $Y$. A resolution, $X$, is a \emph{minimal resolution} of $Y$ if any other resolution factors through it. In general minimal resolutions do not exist, but they always exist for surfaces, \cite[Corollary 27.3]{LipmanRationalSingularities}. A resolution, $X$, is a \emph{crepant resolution} of $Y$ if $\pi^* \omega_Y =\omega_{X}$, where $\omega_{X}$ and $\omega_Y$ are the canonical classes of $X$ and $Y$ which we assume are normal. In general crepant resolutions do not exist. A singularity, $Y$, is \emph{rational} if for any resolution $\pi:X\rightarrow Y$
\begin{equation*}
\mathbb{R} \pi_{*} \mathcal{O}_{X} \cong \mathcal{O}_{Y}.
\end{equation*}
If this holds for one resolution it holds for all resolutions, \cite[Lemma 1]{Viehweg}.
\end{resolution}

Minimal resolutions of rational affine singularities $\pi:X \rightarrow \Spec(R)$ satisfy the condition $\mathbb{R}\pi_* \mathcal{O}_{X} \cong \mathcal{O}_R$ by definition, and in the case of surface singularities it is immediate that the dimensions of the fibres of $\pi$ are $\le 1$. Hence the following corollary is immediate from Corollary \ref{BigTheorem1} ii).

\begin{RationalSurfaceSing} \label{RationalSurfaceSing} Suppose that $\pi:X \rightarrow \Spec(R)$ is the minimal resolution of a rational surface singularity. Then there is a tilting bundle $T_0$ on $X$ as in Theorem \ref{TiltingBundles1dim}, and by Corollary \ref{BigTheorem1} ii) $X$ is the fine moduli space of quiver representations of $A_0=\End_{X}(T_0)^{\op}$ for dimension vector $d_{T_0}$ and stability condition $\theta_{T_0}$ with tautological bundle $T_0^{\vee}$.
\end{RationalSurfaceSing}

This gives a moduli interpretation of minimal resolutions for all rational surface singularities. In certain examples the tilting bundles and algebras are well-understood and this corollary recovers previously known examples.

\begin{KleinianSingularitiesExample}[Kleinian Singularities] \label{Kleinian} Kleinian singularities are quotient singularities $\mathbb{C}^2/G$ for $G$ a non-trivial finite subgroup of $\SL_2(\mathbb{C})$. These have crepant resolutions, and in particular $\Hilb^{G}(\mathbb{C}^2)=X \rightarrow \mathbb{C}^2/G$ is a crepant resolution, \cite{ItoNakamura}. There is a tilting bundle $T$ on $X$ constructed by Kapranov and Vasserot \cite{KapranovVasserot}, which, if we take the multiplicity free version, matches the $T_0$  of Theorem \ref{TiltingBundles1dim}. Then $A=\End_{X}(T)^{\op}$ is presentable as the McKay quiver with relations, the preprojective algebra, and $G$-$\Hilb(\mathbb{C}^2)$ is the quiver GIT quotient of the preprojective algebra for stability condition $\theta_T$ and dimension vector $d_{T}$. The crepant resolutions were previously constructed as hyper-K\"ahler quotients by Kronheimer \cite{KronheimerALE}, this approach was interpreted as a GIT quotient construction by Cassens and Slodowy \cite{CassensSlodowy}, and as a quiver GIT quotient by Crawley-Boevey \cite{ExceptionalFiberCrawleyBoevey}.
\end{KleinianSingularitiesExample}

\begin{SurfaceQuotientSingularitiesExample}[Surface Quotient Singularities] \label{Reconstruction}

As an expansion of the previous example we consider $G$ a non-trivial, pseudo-reflection-free, finite subgroup of $ \GL_2(\mathbb{C})$. Then $\mathbb{C}^2/G$ is a rational singularity with a minimal resolution $\pi: G$-$\Hilb(\mathbb{C})=X \rightarrow \mathbb{C}^2/G$ by \cite{HilbertGL2Ishii}. The variety $X$ has the tilting bundle $T_0$, and the algebras $A=\End_{X}(T_0)^{\op}$ can be presented as the path algebras of  quivers with relations, the reconstruction algebras, which are defined and explicitly calculated in \cite{WemyssGL2,RCAA,RCAD1,RCAD2}. If $G < \SL_2(\mathbb{C})$ then this example falls into the case of Kleinian singularities above, otherwise these fall into a classification in types $\mathbb{A},\mathbb{D},\mathbb{T},\mathbb{I}$, and $\mathbb{O}$, \cite[Section 5]{WemyssGL2}.  It was shown by explicit calculation in \cite{RCAA,RCAD1,RCAD2} that in types $\mathbb{A}$ and $\mathbb{D}$ the minimal resolutions $X$ are quiver GIT quotients of $A$ with stability condition $\theta_{T_0}$ and dimension vector $d_{T_0}$. Corollary \ref{RationalSurfaceSing} recovers these cases without needing to perform explicit calculations, and also includes the same result for the remaining cases $\mathbb{T},\mathbb{I}$, and $\mathbb{O}$.
\end{SurfaceQuotientSingularitiesExample}

\begin{ReconstructionAlg} \label{ReconstructionAlg} Suppose $G< \GL_2(\mathbb{C})$ is a finite, non-trivial, pseudo-reflection-free group. Then the minimal resolution of the quotient singularity $\mathbb{C}^2/G$ can be constructed as the fine moduli space of the quiver representation moduli functor of the corresponding reconstruction algebra for stability condition $\theta_{T_0}$ and dimension vector $d_{T_0}$, and the tautological bundle is the tilting bundle $T_0^{\vee}$.
\end{ReconstructionAlg}
\begin{proof}

We note that in Theorem \ref{TiltingBundles1dim} $T_1=T_0^{\vee}$ and that $\End_{X}(T_0^{\vee}) \cong \End_{X}(T_0)^{\op}$. Hence our definition of $A=\End_{X}(T_0)^{\op}$ as the reconstruction algebra matches that given in \cite{WemyssGL2,RCAA,RCAD1,RCAD2} as $A=\End_{X}(T_1)$. Then the result is an immediate corollary of Corollary \ref{RationalSurfaceSing}.
\end{proof}

\begin{DetermentalSingularitiesExample}[Determinantal Singularities]
We give one higher dimensional example. Let $R$ be the $\mathbb{C}$-algebra $\mathbb{C}[X_0, \dots X_l, Y_1, \dots Y_{l+1}]$ subject to the relations generated by all two by two minors of the matrix
\begin{equation*}
\left( \begin{array}{c c c c c c} X_0 & X_1 & \dots & X_i & \dots & X_{l} \\
Y_1 & Y_2 & \dots & Y_{i+1} & \dots & Y_{l+1}
\end{array} 
\right).
\end{equation*}
Then $\Spec(R)$ is a $l+2$ dimensional rational singularity and has an isolated singularity at the origin. This has a resolution given by $\pi: X= \Tot \left( \bigoplus_{i=1}^l \mathcal{O}_{\mathbb{P}^1}(-1) \right) \rightarrow \Spec(R)$, the total space of the locally free sheaf $\bigoplus_{i=1}^l\mathcal{O}_{\mathbb{P}^1}(-1)$ mapping onto its affinisation. The variety $X$ has a tilting bundle $T_0$ by Theorem \ref{TiltingBundles1dim}, which, considering the bundle map $f:X \rightarrow \mathbb{P}^1 $, we can identify as $T_0=\mathcal{O}_{X} \oplus f^* \mathcal{O}_{\mathbb{P}^1}(-1)$. We can then present $A_0=\End_{X}(T_0)^{\op}$ as the following quiver with relations, $(Q,\Lambda)$.
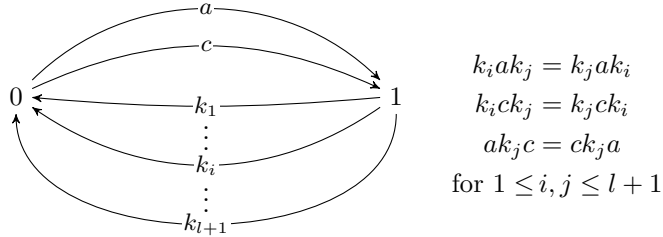
\begin{center}
\begin{align*}
\begin{aligned}
\begin{tikzpicture} [bend angle=45, looseness=1]
\node (C1) at (0,0)  {$0$};
\node (C2) at (5,0)  {$1$};
\node (C3) at (2.5,-0.45)  {$\vdots$};
\node (C3) at (2.5,-1.27)  {$\vdots$};
\draw [->,bend left] (C1) to node[gap]  {\small{$a$}} (C2);
\draw [->,bend left=25] (C1) to node[gap] {\small{$c$}} (C2);
\draw [->,bend left=90] (C2) to node[gap]  {\small{$k_{l+1}$}} (C1);
\draw [->,bend left=5] (C2) to node[gap] {\small{$k_1$}} (C1);
\draw [->,bend left=35] (C2) to node[gap] {\small{$k_i$}} (C1);
\end{tikzpicture}
\end{aligned}
& \quad
\begin{aligned}
k_iak_j &= k_jak_i \\
k_ick_j&=k_j c k_i \\
ak_j c &= c k_j a  \\
 \text{ for } 1 \le & \, i,j  \le l+1  
\end{aligned}
\end{align*}
\end{center}

By Theorem \ref{BigTheorem1} we know that $X$ can be reconstructed as the quiver GIT quotient of $A_0$ with dimension vector $d_{T_0}=(1,1)$ and stability condition $\theta_{T_0}=(-1,1)$. In this example we will explicitly verify this. A dimension $d_{T_0}$ representation is defined by assigning a value $\lambda_i \in \mathbb{C}$ to each $k_i$ and $(\alpha,\gamma) \in \mathbb{C}^2$ to $(a,c)$. The relations are all automatically satisfied so $\Rep_{d_{T_0}}(Q,\Lambda)= \mathbb{C}^{l+1} \times \mathbb{C}^2$.  Then a representation is $\theta_{T_0}$ stable if it has no dimension $(1,0)$ submodules, so these correspond to the subvariety with  $(\alpha,\gamma) \in \mathbb{C}^2/(0,0)$, hence $\Rep_d(Q,\Lambda)^{ss}= \mathbb{C}^{l+1} \times \mathbb{C}^2/(0,0)$. We then find that the corresponding quiver GIT quotient is given by the action of $\mathbb{C}^*$ on $ \mathbb{C}^{l+1}  \times \mathbb{C}^2/(0,0) $ with weights $-1$ on  $\mathbb{C}^2/(0,0)$ and $1$ on $\mathbb{C}^{l+1}$. This produces the total bundle $X$.

When $l=2$  this is the motivating example of the Atiyah flop given as the opening example of \cite{3Dflops} and $A_0$ is the conifold quiver. In this case, by Theorem \ref{FlopGIT}, we can calculate the flop as the quiver GIT quotient of $A_1 \cong A_0^{\op}$.
\end{DetermentalSingularitiesExample}

\appendix

\section{Comparing Quiver Moduli Functors} \label{QuiverModuliForAlgebra}
As we noted in the introduction, our results are inspired by a theorem of Sekiya and Yamaura that compares quiver GIT quotients for algebras related by tilting modules. This is done by constructing natural transformations between moduli functors which should have the quiver GIT quotients as moduli spaces, however the quiver representation moduli functor considered in \cite{SekiyaYamaura} is different to the one defined in Section \ref{Quiver GIT moduli functors} and does not always have the quiver GIT quotient as a moduli space. In this appendix we outline the minimal changes required to reinterpret the results of \cite{SekiyaYamaura} for a correct moduli functor.

The moduli functor for quiver representations defined  in \cite[Section 4.1]{SekiyaYamaura} is 
\begin{align*}
 \mathcal{F}_{A,d,\theta}^{SY}: & \, \mathfrak{R} \rightarrow \mathfrak{Sets} \\
    & R  \mapsto \left. \mathcal{S}^{ss}_{A,d,\theta}(R) \middle/ \text{$\sim_{SY}$} \right.
\end{align*} 
with the set $\mathcal{S}^{ss}_{A,d,\theta}(R)$ defined as in Section \ref{Quiver GIT moduli functors} and the equivalence condition $M_1 \sim_{SY} M_2$ if $M_1 \otimes_R R/ \mathfrak{m}$ is $S$-equivalent to $M_2 \otimes_R R/\mathfrak{m}$ for all $\mathfrak{m} \in \MaxSpec(R)$. This differs from the functor $\mathcal{F}_{A,d,\theta}^{ss}$ defined in Section \ref{Quiver GIT moduli functors} by using the equivalence $\sim_{SY}$ rather than the equivalence $\sim$. However, as the following example shows, the equivalence $\sim_{SY}$ is too restrictive.

\begin{EquivalenceExample} Let $A=\mathbb{C}[x]$, $d=1$, and $\theta=0$. Then $A$ can be presented as the path algebra of a quiver with a single vertex and single loop, and the quiver GIT quotient is $\Spec(A)$. In particular, if this were a fine moduli space for $\mathcal{F}_{A,d,\theta}^{SY}$ then $\mathcal{F}_{A,d,\theta}^{SY}(\mathbb{C}[\epsilon]/\epsilon^2) \cong \Hom_{\mathfrak{Sch}}(\Spec \, \mathbb{C}[\epsilon]/\epsilon^2, \Spec \,A)  \cong \mathbb{C}^2$. However \[  \{ M_{a,b}:=\mathbb{C}[x,\epsilon]/(x- a-b \epsilon) \,  | \, a,b \in \mathbb{C} \} = \mathcal{S}_{A,d,\theta}^{ss}(\mathbb{C}[\epsilon]/\epsilon^2)  \]
and $M_{a,b} \sim_{SY} M_{\alpha,\beta} \Leftrightarrow a=\alpha$ so $\mathcal{F}_{A,d,\theta}^{SY}(\mathbb{C}[\epsilon]/\epsilon^2) \cong \mathbb{C}$. Hence the quiver GIT quotient is not a fine moduli space for the functor $\mathcal{F}_{A,d,\theta}^{SY}$.
\end{EquivalenceExample}

This indicates that $\sim_{SY}$ is not the correct equivalence to use to define a quiver representation moduli functor. Below we note a brief amendment that adapts the results of \cite{SekiyaYamaura} to work with the functor used in this paper instead.

Firstly, the moduli functor defined in \cite[Section 4.1]{SekiyaYamaura} should be replaced by the moduli functor $\mathcal{F}_{A,d,\theta}^{ss}$ defined in Section \ref{Quiver GIT moduli functors} and the statement that $\mathcal{M}_{d,\theta}^{ss}$ is a coarse moduli space can then be replaced by the statement that $\mathcal{F}_{A,d,\theta}^{ss}$ is corepresented by $\mathcal{M}_{d,\theta}^{ss}$ and when $d$ is indivisible and $\theta$ generic this is a fine moduli space.

There are then minimal changes to make; the majority of the work in \cite{SekiyaYamaura} concerns only the sets $\mathcal{S}_{A,d,\theta}^{ss}(R)$ so needs no alteration.  The moduli functor enters the results via \cite[Proposition 4.5]{SekiyaYamaura}, which gives conditions for a family of functors $F^R:\mathcal{S}_{B,d',\theta'}^{ss}(R) \rightarrow \mathcal{S}_{A,d,\theta}^{ss}(R)$ to define a natural transformation between quiver representation moduli functors and shows that such a natural transformation induces a morphism of schemes between the quiver GIT quotients. A natural transformation of functors induces a morphism between corepresenting schemes by the universal property, and to adapt the conditions for a family to induce a natural transformation for the moduli functor with equivalence $\sim$ rather than $\sim_{SY}$ we need only add an additional condition to ensure that the natural transformation is well defined under the equivalence $\sim$:
\[
 \quad  F^R(M \otimes_R L) \cong F^R(M) \otimes_R L \text{ for any invertible $R$-module $L$ and $M \in \mathcal{S}_{A,d,\theta}^{ss}(R).$}
\]
The only other results which involve the moduli functor are \cite[Theorems 4.6 and 4.11]{SekiyaYamaura} which check that the conditions of \cite[Proposition 4.5]{SekiyaYamaura} are satisfied by the specific functors $\Hom_{A^R}(T^R,-)$ and $T^R \otimes_{A^R} (-)$ when $T$ has a finite length resolution by projective $A$-modules, and \cite[Theorem 4.20]{SekiyaYamaura} which combines these two results in the case where $T$ is a tilting module. It is easy to see that the functors  $\Hom_{A^R}(T^R,-)$ and $T^R \otimes_{A^R} (-)$ also satisfy the additional condition: this follows from \cite[Lemmas 4.7 and 4.14]{SekiyaYamaura} in the case of an invertible $R$-module. As such the main result \cite[Theorem 4.20]{SekiyaYamaura} holds when the moduli functor is taken to be $\mathcal{F}^{ss}_{A,d,\theta}$ rather than $\mathcal{F}_{A,d,\theta}^{SY}$.

\begin{SekiyaYamauraCorrection}[{\cite[Theorem 4.20]{SekiyaYamaura}}] Let $B$ be an algebra with tilting module $T$. Define $A=\End_{B}(T)^{\op}$, suppose that both $A$ and $B$ are presented as  path algebras of quivers with relations, and let $\mathcal{F}^{ss}_{A,d,\theta}$ and $\mathcal{F}^{ss}_{B,d',\theta'}$ denote quiver representation moduli functors on $A$ and $B$ for some choice of dimension vectors $d,d'$ and stability conditions $\theta, \theta'$. Then if the tilting equivalences 
\begin{center}
\[
\begin{tikzpicture} [bend angle=15, looseness=1]
\node (C1) at (0,0)  {$D^b(B$-$\mod)$};
\node (C2) at (4,0)  {$D^b(A$-$\mod)$};
\draw [->,bend left] (C1) to node[above]  {$\scriptstyle{\mathbb{R}\Hom_{B}(T,-)}$} (C2);
\draw [->,bend left] (C2) to node[below]  {\scriptsize{$ T\otimes ^{\mathbb{L}}_{A}(-)$}} (C1);
\end{tikzpicture}
\]
\end{center}
 restrict to a bijection between $\mathcal{F}^{ss}_{B,d',\theta'}(\mathbb{C})$ and $\mathcal{F}^{ss}_{A,d,\theta}(\mathbb{C})$ then $\mathcal{F}^{ss}_{B,d',\theta'}$ is naturally isomorphic to $\mathcal{F}^{ss}_{A,d,\theta}$. Hence by the universal property of corepresenting schemes the corresponding quiver GIT quotients are isomorphic as schemes.
\end{SekiyaYamauraCorrection}

\bibliographystyle{plain}
\bibliography{QGITBIB}
\end{document}